\documentclass[reqno]{amsart}
\usepackage{geometry}                
\usepackage[parfill]{parskip}    
\usepackage{graphicx}
\usepackage{amssymb}
\usepackage{epstopdf}
\usepackage{stmaryrd}
\usepackage{esint}
\usepackage[T1]{fontenc}
\usepackage{mathrsfs} 
\usepackage[english]{babel}
\usepackage{amssymb,
enumerate}
\usepackage{color}
\usepackage[latin1]{inputenc}
\usepackage[T1]{fontenc}
\usepackage{amsfonts}
\usepackage{amssymb}
\usepackage{amsmath}
\usepackage{amsthm}
\usepackage{graphicx,xcolor}
\usepackage{afterpage}
\usepackage[colorlinks=true, linkcolor=red, citecolor=cyan, urlcolor=green]{hyperref} 
\usepackage{bbm}
\usepackage{latexsym}

\def\be{\begin{equation}}
\def\ee{\end{equation}}
\def\bea{\begin{eqnarray}}
\def\eea{\end{eqnarray}}

\def\nn{\nonumber}

\def\Ham{\operatorname{Hamming}}
\def\R{\mathbb{R}}
\def\s{\sigma}
\def\a{\alpha}
\def\e{\varepsilon}

\def\b{\beta}

\def\t{\tau}

\def\sign{\operatorname{sign}}

\def\card{\operatorname{Card}}

\newcommand{\ie}{\textit{i.e. }}

\newcommand{\OO}[1]{O \left(\frac{1}{#1}\right)}

\newcommand{\wh}{\widehat}

\DeclareMathSymbol{\leqslant}{\mathalpha}{AMSa}{"36} 
\DeclareMathSymbol{\geqslant}{\mathalpha}{AMSa}{"3E} 
\DeclareMathSymbol{\eset}{\mathalpha}{AMSb}{"3F}     
\renewcommand{\leq}{\;\leqslant\;}                   
\renewcommand{\geq}{\;\geqslant\;}                   

\DeclareMathOperator{\Barr}{Bar}

\DeclareMathOperator{\Var}{Var}

\def\ie{\textit{i.e. }}

\def\a{\alpha}
\def\e{\varepsilon}

\def\b{\beta}

\def\l{\lambda}

\def\s{\sigma}
\def\t{\tau}

\def\R{\mathbb{R}}

\theoremstyle{plain}
\newtheorem{theorem}{Theorem}[section]
\newtheorem{lemma}{Lemma}[section]
\newtheorem{proposition}{Proposition}[section]

\theoremstyle{definition}
\newtheorem{remark}{Remark}[section]

\numberwithin{equation}{section}

\definecolor{light}{gray}{.9}

\title[Retrieval in RBMs]
{Pattern reconstruction with restricted Boltzmann machines}

\author{Giuseppe Genovese}
\address{Institute of Mathematics, University of Zurich, Winterthurerstrasse 190, 8057 Zurich, Switzerland.}
\email{giuseppe.genovese@math.uzh.ch}

\date{\today}                     

\begin{document}
\maketitle

\begin{abstract}
Restricted Boltzmann machines are energy models made of a visible and a hidden layer. We identify an effective energy function describing the zero-temperature landscape on the visible units and depending only on the tail behaviour of the hidden layer prior distribution. Studying the location of the local minima of such an energy function, we show that the ability of a restricted Boltzmann machine to reconstruct a random pattern depends indeed only on the tail of the hidden prior distribution. We find that hidden priors with strictly super-Gaussian tails give only a logarithmic loss in pattern retrieval, while an efficient retrieval is much harder with hidden units with strictly sub-Gaussian tails; if the hidden prior has Gaussian tails, the retrieval capability is determined by the number of hidden units (as in the Hopfield model). 
\end{abstract}

\section{Introduction}\label{sect:intro}

Restricted Boltzmann machines (RBMs) are represented by probability distributions on the product space $\{-1,1\}^{N_1}\times \R^{N_2}$ whose density w.r.t. the uniform probability on $\{-1,1\}^{N_1}$ times some prior distribution on $\R^{N_2}$ depends on a matrix valued parameter $W$ (so-called weight matrix) and it is proportional to
\be\label{eq:RBMp}
\exp\left(\sum_{i\in[N_2]}\sum_{\mu\in[N_2]}\s_iW_i^{\mu} z_\mu\right)\,.
\ee
The vectors $(\s_1,\ldots,\s_{N_1})$ and $(z_1,\ldots,z_{N_2})$ are called respectively visible and hidden layer and their entries visible and hidden units. Typically the units are i.i.d.

RBMs are widely studied generative models of machine learning, introduced long ago in \cite{smo}. Their mathematical relation with models of associative memory, such as the ones proposed by Little \cite{little} or Hopfield \cite{hop}, was noted at the very early stage of the theory, see \cite{hinton}. Indeed exploiting the product structure of the RBM distribution integrating out the hidden layer, one can analyse the corresponding model of associative memory, see \cite{tubiana}, \cite{PRE}, \cite{bulso}; the simplest example is the Hopfield model, obtained by a RBM with Gaussian hidden prior by a Hubbard-Stratonovic transform. The present work aligns with this line of research.

We will be interested in particular to the possibility of reconstructing a pattern just looking at the typical configurations of the visible layer of a RBM. This operation is also called pattern retrieval.  
The question is relevant for the understanding of how the hidden layer affects the configurations of the visible units in the RBM distribution. In the last years the main focus on RBMs has been on learning the unknown probability distribution underlying a given dataset \cite{bresler19, goel20, back, nature}. For the practitioner, learning a RBM amounts to fitting the true law of the data by tuning the weights in the density (\ref{eq:RBMp}) and this is typically done by gradient ascent on the Kullback-Leibler divergence. After the learning process, deep local minima of the energy function are supposed to fall close to the datapoints. Therefore the analysis of learning in RBMs consists of two tasks: understanding the complex landscape of the energy in the vicinity of the datapoints at given weights and devising good optimisation algorithms to fit the data. The investigation of each of these steps is a true mathematical challenge. In these respects pattern retrieval represents a simplified setting to study at first instance, as the roles of the datapoints and of the weights is undertaken by the same objects, the patterns. So there is no optimisation, but one only has to look at the energy landscape in the vicinity of the patterns. 

More precisely, we study the retrieval of i.i.d. binary patterns as the distribution of the hidden layer varies. We do it looking at the local minima of the energy function, in what is called in statistical physics a zero temperature limit, in which retrieval is maximised (see e.g. \cite{amit}).
We prove that the tail of the hidden prior distribution determines the retrieval capability of RBMs. More precisely, for priors with tails ranging from exponential to Gaussian we prove that deep local minima are well localised about the patterns, while if the tails of the hidden prior decay faster than Gaussian, we show that the patterns cannot be retrieved well in any case. RBMs whose hidden priors have Gaussian tails (a class including the Hopfield model) represent special threshold cases which we treat separately in either the positive (Theorem \ref{TH1}) and the negative (Theorem \ref{TH2}) result below. 

\subsection{Setting}

We consider RBMs with i.i.d. Bernoulli $\pm1$ visible units $\s_1,\ldots,\s_{N_1}$ and symmetric i.i.d. hidden units $z_1,\ldots, z_{N_2}$ distributed according to some prior $\pi$. We allow a certain freedom in the choice of the hidden prior $\pi$, for which we only require that
\be\label{eq:priorZ}
\pi(|z|\geq t)\simeq e^{-|t|^{q}}\qquad\mbox{for some $q>1$}\,. 
\ee

Let $\xi^{(1)},\ldots, \xi^{(N_2)}$ denote independent random vectors, that we call patterns, with $N_1$ centred i.i.d. $\pm1$ components. We consider RBM probability distributions with (unnormalised) density (w.r.t. the priors)
\be\label{eq:p}
p(\s,z;\xi):= \exp\left(\frac{\b}{N_1^{\frac 1{q_-}}}\sum_{\mu\in[N_2]}(\s,\xi^{(\mu)}) z_\mu\right)\,,
\ee
where $\b>0$ is a parameter usually called {\em inverse temperature} and $q_-:=\min(q,2)$. We shall consider the ratio between the number of visible and hidden units as follows
\be\label{eq:alpha}
\a:=\frac{N_2}{N_1^\frac{p_+}{2}}\,,
\ee 
where $\frac1{p_+}+\frac1{q_-}=1$. This is a parameter which will be considered as a constant in the subsequent analysis.  
The normalisation factors in (\ref{eq:p}) and (\ref{eq:alpha}) are unusual. For instance in \eqref{eq:p} typically from a spin glass perspective one adopts a more familiar normalisation with $\sqrt N_1$, while for learning one leaves the energy unnormalised (as the best normalisation is learned with the weights). Our choice ensures that either the energy of the single pattern and the global maximum of (\ref{eq:p}) as $\b\to\infty$ stay bounded as $N_1$ grows to infinity and scale linearly with $\a$ (with constants depending on $q$). The aim of Section \ref{sect:T0} is to make this point more precise. 
 
Integrating out the hidden layer in (\ref{eq:p}) we get a probability distribution over the visible units. Its density writes as
\be\label{eq:ditrsRHS}
\int p(\s,z;\xi) \pi(dz_1)\ldots \pi(dz_{N_2})=\exp\left(\sum_{\mu\in[N_2]}u\left(\frac\b{N_1^{\frac 1q}}(\xi^{\mu},\s)\right)\right)\,,
\ee
where 
$$
u(x):=\log E[e^{xz_1}]\,.
$$
We are interested in studying the local maxima of the r.h.s. of \eqref{eq:ditrsRHS} as $\b$ is very large, but $N_1,N_2$ finite. The main issue is that the dependency on $\b$ in the exponent is not multiplicative and it is not clear which function should be analysed in the limit $\b\to\infty$ (compare it for instance with the easier cases of the Sherrington-Kirkpatrick model \cite{SK} or the Hopfield model \cite{newman}, where $\b$ is just a multiplicative parameter). 

Exploiting a reduction argument introduced in \cite{PRE}, we show how to single out an effective energy function which captures the RBM landscape at zero temperature. The following simple observation starts our considerations: for any $z_1$ such that \eqref{eq:priorZ} holds for some $q>1$, it is
\be\label{eq:start}
c(q) \|z_1\|_{\psi_q}^p\leq \lim_{x\to\infty} \frac{u(x)}{|x|^{p}}\leq C(q) \|z_1\|_{\psi_q}^p\,
\ee
where $0<c(q)\leq C(q)<\infty$ are universal constants depending only on $q$ and $p$ is the H\"older conjugate exponent of $q$. For a definition of the Orlicz norms $\|\cdot\|_{\psi_q}$, see (\ref{eq:Orl}) below.
The proof of \eqref{eq:start} is immediate: we write
\bea
E[e^{xz_1}]&=&\int_0^\infty d\l P(z_1\geq x^{-1}\log\l)\simeq\int_0^\infty d\l e^{-\frac{|\log\l|^q}{\|z_1\|^q_{\psi_q}|x|^q}}\nn\\
&=&\|z_1\|_{\psi_q}|x|\int_0^\infty d\l e^{\l \|z_1\|_{\psi_q}|x|-|\l|^q}\nn\\
&=&\|z_1\|_{\psi_q}|x|e^{\frac{\|z_1\|_{\psi_q}^p|x|^p}{p}}\int_0^\infty d\l e^{\l\|z_1\|_{\psi_q} x-|\l|^q-\frac{\|z_1\|_{\psi_q}^p|x|^p}{p}}\nn\,,
\eea
and the last integral is finite uniformly in $x$ by the Young inequality. Taking the $\log$ on both sides and passing to the limit we get (\ref{eq:start}).

Thus by \eqref{eq:start}, as $\b\to\infty$ we are led to consider the following $p$-spin energy function \cite{gardner,bov-p}:
\be\label{eq:Hp}
H^{(p)}(\s;\xi):=-\frac{1}{N_1^{\kappa(p)}}\sum_{\mu\in[N_2]}|(\xi^{(\mu)},\s)|^p\,,\qquad \kappa(p):=1+p-\frac{p}{p_+}\,
\ee
(here we include the usual normalisation factor $1/N_1$ of the internal energy directly in the definition of $H^{(p)}$). 
To fix the ideas, $\kappa(p)=p$ for $p\geq 2$ and $\kappa(p)=1+\frac{p}{2}$ for $p\leq2$.

\subsection{Main results}

The focus of this paper is to study the location of the minima of \eqref{eq:Hp} on $\{-1,1\}^{N_1}$ close to the pattern configurations, in the limit $N_1,N_2\to\infty$ while $\a$ remains constant. Hence the main object of interest will be the following two sets.

\begin{align}
\mathtt{LM}_{N_1}&:=&\{\mbox{local minima of \eqref{eq:Hp}}\}\label{eq:LMset-1}\,,\\
\mathtt{dLM}_{N_1}^{(\mu)}&:=&\{\mbox{local minima $\bar\s$ of (\ref{eq:Hp}), s.t. $H^{(p)}(\xi^{(\mu)};\xi)-H^{(p)}(\bar\s;\xi)>0$}\}\label{eq:DLMset-1}\,.
\end{align}

Below $\Ham(a,b)$ denotes the Hamming distance between $a,b$ (\ie the number of different entries) and $\widehat B^{N_1}_{\mu,R}$ is the ball in this metric centred at the $\mu$-th pattern with radius $R$. Throughout the paper we will repeatedly use that two patterns are typically separated by $N_1/2$ flips, so that $\widehat B^{N_1}_{\mu,\lfloor N_1/2\rfloor}$ and $\widehat B^{N_1}_{\mu',\lfloor N_1/2\rfloor}$ typically do not overlap. We say that the event $A$ occurs with high probability (w.h.p.) if for all $x>0$, for all sufficiently large $N_1$ it holds $P(A)\geq 1-N_1^{-x}$. $S\,:\,[0,1]\mapsto\R$ denotes the coin tossing entropy
\be\label{eq:defS}
S(r):=-r\log r-(1-r)\log(1-r)\,.
\ee

Our first result states that the error of reconstructing a given pattern is very small in terms of the number of visible units $N_1$ if the decay of the hidden prior (\ref{eq:priorZ}) is slower than Gaussian, while a finite fraction of bits cannot be retrieved for $q=2$.

\begin{theorem}\label{TH1}
Let $q\in(1,2)$. There exists $r_0\in(0,\frac12]$ such that w.h.p.
\be\label{eq:TH1-p>2}
\max_{\mu\in[N_2]}\max_{\s\in \mathtt{dLM}_{N_1}^{(\mu)}\cap \widehat B^{N_1}_{\mu,\lfloor r_0N_1\rfloor}}\Ham\left(\s, \xi^{(\mu)}\right)\leq (\log N_1)^{\frac{q-1}{2-q}}\,.
\ee
Let $q=2$. For any $r\in(0,3/8)$ if $\a< \min\left(\frac13\sqrt \frac{r}{1-r},\frac{\sqrt r}{25S(r)}\right)$ then w.h.p.
\be\label{eq:TH1-p=2}
\max_{\mu\in[N_2]}\max_{\s\in \mathtt{dLM}_{N_1}^{(\mu)}\cap \widehat B^{N_1}_{\mu,\lfloor 3N_1/8\rfloor}}\Ham\left(\s, \xi^{(\mu)}\right)\leq  rN_1\,.
\ee
\end{theorem}

Theorem \ref{TH1} identifies the models with $q<2$ as excellent in pattern reconstruction: there are deep minima located few flips away from the patterns (in fact polylog flips, see (\ref{eq:TH1-p>2})) and no deeper minima appear in an extended region. We observe that albeit we formulate Theorem \ref{TH1} in terms of local minima, we proved a stronger statement regarding all points in a Hamming ball about the pattern. Namely we show that exploring all the points in a large Hamming ball centred at any pattern, to find a point with lower energy we need to go very close to the centre. 

We do not attempt here at precisely characterising the radius $r_0$, the basin of attraction of the patterns, the maximal $\a$ allowing retrieval or any of the constants in the play. Indeed the numbers appearing in the case $q=2$ of the above Theorem carry no special meaning.

Local minima are not directly related to the typical configurations of (\ref{eq:p}). However it is well known that any algorithmic search of typical configurations will finish to find a hopefully representative local minimum. This can be done by the usual {\sc flip} algorithm, that is greedy flipping of one units at time decreasing the energy until no more decreasing is possible. Therefore $\mathtt{dLM}_{N_1}^{(\mu)}$ has a direct interpretation in terms of retrieval. Take for instance $q<2$. By the proof of Theorem \ref{TH1} it follows that any {\sc flip} search initialised for instance at $\xi^{(\mu)}$ will end up in a point of $\mathtt{dLM}_{N_1}^{(\mu)}$ falling only $(\log N_1)^{\frac{q-1}{2-q}}$ flips away from the pattern, which means that only few bits are mis-retrieved. 

Somewhat in the opposite direction, the next result shows that for $q>2$ in (\ref{eq:priorZ}), the local minima of (\ref{eq:Hp}) are quite far from the patterns. 
\begin{theorem}\label{TH2}
Let $q\geq2$, $r\in[0,\frac12]$ and let $\a_q(r):=S(r)$ for $q\neq2$ and $\a_2(r):=S(r)/(1-2r)^2$. There is a numerical constant $f(q)>0$ such that every $r\in(0,\frac12)$ and for all $\a\geq f(q)\a_q(r)$ we have w.h.p.
\be\label{eq:TH2}
\Ham\left(\mathtt{LM}_{N_1}, \xi^{(\mu)}\right)\geq \lfloor rN_1\rfloor\,.
\ee
\end{theorem}

For sake of brevity the value of the numerical constant $f(q)$ is not specified in the statement of the previous theorem, but can be determined following its proof. Again we stress that we did not aim at optimising the constants.

According to Theorem \ref{TH2}, if $q>2$ one could still hope for retrieval with a very small amount of hidden variables, \ie for $\a$ small enough (indeed for $\a=0$ reconstruction is possible, see \cite{PRE}), but for $\a$ larger than a given constant no recovery is allowed. For $q=2$ the situation improves a bit in the sense that pattern reconstruction becomes less and less efficient as $\a$ grows. 

The paper \cite{PRE} showed that RBMs with hidden prior interpolating between a Gaussian and a bimodal symmetric distribution exhibit retrieval at finite $\a>0$, which disappears in the degenerate case when the Gaussian part is switched off. It is also argued that such a lack of retrieval should persist at least for any compactly supported hidden prior. This is demonstrated using non-rigorous replica computations and numerics. We give here the first mathematical confirm of these findings, as Theorem \ref{TH1} (for $q=2$) implies pattern retrieval if in the interpolating prior the Gaussian part is present, whatever small, and Theorem \ref{TH2} is a strong indication for lack of retrieval for hidden prior with a Bernoulli $\pm1$ distribution (for which we should read $q=\infty$). 

\subsection{Related literature}

The results here presented mark a neat difference in the retrieval capabilities of RBMs with hidden priors (\ref{eq:priorZ}) with $q<2$ (very good capabilities) and $q>2$ (not so good) with a transition at the Gaussian tail case $q=2$. As already remarked, a notable instance of the case $q=2$ is the Hopfield model, for which a similar analysis at zero temperature was done in \cite{newman} (analog of Theorem \ref{TH1}), \cite{lou} (analog of Theorem \ref{TH2}) and \cite{talahop} in the attempt of proving the picture of \cite{amit}. When comparing these papers to ours, we underline that we do not seek to characterise any of our estimates with the best possible constants, which was instead a relevant component of all these previous papers. In particular by Theorem \ref{TH1} it follows that in the case $q=2$ we observe retrieval for $\alpha\leq 0.04$, much less than the threshold $\a\leq 0.14$ computed by Amit, Gutfreund and Sompolinsky. However from our analysis it is clear that this critical threshold is not a specific of the Hopfield model, but it can be achieved universally for all the models whose hidden priors has Gaussian tails.

We exploit and make mathematically precise the heuristics of \cite{PRE}. Namely we use that the tail of the hidden prior determines the behaviour for large argument of the energy function of the associative network (around zero it is always quadratic). It is exactly this asymptotic that governs retrieval: the more convex the better. 
Mathematically speaking the introduction of the hidden layer is a way to linearise the energy function (over the visible units) and different prior distributions for the hidden layer correspond to different associative networks. Similar ideas have been used by \cite{tubiana}, \cite{PRE}, \cite{cocco}, \cite{bulso}, \cite{treves} to study the performance of the RBMs with varying hidden unit statistics. 

The {\sc flip} algorithm is a very natural choice to explore the energy landscape of RBMs and indeed it is essentially the original network dynamics proposed in \cite{hop}. This gives a nice connection with the local max-cut problem as analysed for instance in \cite{maxcut2} and \cite{maxcut1}, even though here we exploit crucially the presence of the patterns, which constitute a special class of local minima. This is even more clear by comparing with the analysis for the Sherrington-Kirkpatrick model of \cite{SK}. 

Many other dynamics have been proposed alternative to the {\sc flip} algorithm mainly for the Hopfield model and we will not give here an account on that (see the recent work \cite{zamponi} and the references therein). We just mention that the dynamics analysed in \cite{tubiana} and \cite{marinari}, which is a zero-temperature version of the alternate Gibbs sampling typically used to train RBMs, is in spirit very close to our zero-temperature reduction.

\subsection{Notations}\label{sect:nota}

Throughout the paper $p,q\geq1$ will always be H\"older conjugate, that is $\frac 1p+\frac1q=1$ and similarly for $q_-, p_+$, with $q_-:=\min(2,q)$, $p_+:=\max(2,p)$. 
$C,c$ everywhere denote positive absolute constants which may change from formula to formula. We write $X\lesssim Y$ if $X\leq CY$ and $X\simeq Y$ if $Y\lesssim X\lesssim Y$. Sometimes we write $\simeq_a$ or $\lesssim_a$ to stress the dependence of the constants $C$ above on a parameter $a$. We indicate by $(\cdot,\cdot)$ the inner product in $\R^{N_1}$ or $\R^{N_2}$ and the meaning will be always clear from the context and by $\|\cdot\|_p$ the $\ell_p$-norms. $1$ may represent the vector in $\R^{N_1}$ or in $\R^{N_2}$ with all entries equal to $1$. $B_{N}^{(q)}$ is the $\ell_q$ centred ball of radius one in $\R^N$. $\widehat S^{N_1-1}_{\mu,R}, \widehat B^{N_1}_{\mu,R}$ denote respectively the $N_1$-dimensional Hamming ball and sphere centred at $\xi^{(\mu)}$ of radius $R$. If $v\in\R^N$ and $J\subset [N]$ we denote by $v_J$ a vector in $\R^{|J|}$ such that $(v_J)_i=v_{j_i}$ if $J=\{j_1,\ldots, j_{|J|}\}$. To any $J\subset [N]$ we also associate a {\sc flip} operator $F_J$ defined by $(F_J v)_i=-v_i$ if $i\in J$ and $(F_J v)_i=v_i$ if $i\notin J$.
We will use the following Orlicz norms:
\be\label{eq:Orl}
\|Z\|_{\psi_r}:=\inf\left\{\l>0\,:\,E\left[\psi_r\left(\frac{|Z|}{\l}\right)\right]<2\right\}\,,\qquad r>0\,,
\ee
where $\psi_r(x)=e^{x^r}$ for any $x>0$ for $r\geq1$, while for $r\in(0,1)$ there are $c(r),x(r)$ such that for $x\in(0,x(r))$ it is $\psi_r(x)=c(r)x$. 
We underline that, setting $q_s:=\sup\{q'>1\,:\,\|Z\|_{\psi_{q'}}<\infty\}\,$
we have $P(|Z|\geq t)\simeq e^{-|t|^q_s}$ (we convey that bounded random variables have finite $\psi_\infty$-norm).
Bearing in mind the definition (\ref{eq:defS}), we will often use the standard bound for $r\in[0,1]$
\be\label{eq:standard-bound}
\card \widehat S^{N_1-1}_{\mu,\lfloor rN_1\rfloor}= \binom{N_1}{\lfloor rN_1\rfloor}\leq e^{N_1S(r)}\,.
\ee
We denote the transpose patterns $\tilde\xi^{(i)}$ by $\tilde\xi^{(i)}_\mu:=\xi_i^{(\mu)}$, $i\in[N_1]$, $\mu\in[N_2]$. Sometimes we write $\widehat \xi:=\xi/\sqrt{N_1}$. 
$A^c$ is the complement of the set $A$. We say that the event $A$ occurs with high probability (w.h.p.) if for all $x>0$, for all sufficiently large $N_1$ it holds $P(A)\geq 1-N_1^{-x}$. 

\subsection*{Acknowledgements} The author thanks David Belius for many valuable suggestions regarding the presentation of the results. 


\section{Zero temperature reduction}\label{sect:T0}

In this section, which is in part independent on the rest of the paper, we study some interesting properties of the Hamiltonian (\ref{eq:Hp}).

First we show that the single pattern energy is close to the ground state, so providing a motivation for the normalisation factors in (\ref{eq:p}) and (\ref{eq:alpha}). We give a lower bound for the ground state energy linear in $\a$. To do so we do not actually need binary patterns.

\begin{proposition}\label{Prop:groundstate}
Let $\xi^{(1)},\ldots, \xi^{(N_2)}$ be independent vectors in $\R^{N_1}$ with i.i.d. centred sub-Gaussian entries. It holds for $p\geq1$
\be\label{eq:notfar}
\inf_{\s\in \{-1,1\}^{N_1}} H^{(p)}(\s;\xi)\gtrsim_p -(1+\a)
\ee
with probability larger than $1-e^{-c\a^{\frac2p}N_1}$. 
\end{proposition}
To prove Proposition \ref{Prop:groundstate} we need the following auxiliary lemma. 

\begin{lemma}\label{lemma:Psup}Let $\xi^{(1)},\ldots, \xi^{(N_2)}$ be independent vectors in $\R^{N_1}$ with i.i.d. centred sub-Gaussian entries.
Let $p\geq1$. For all $t\gtrsim_p(1+\a)^{\frac1p}$
\be\label{eq:Psup}
P\left(\frac{1}{N_1^{\max(0,\frac{q-2}{2q})}}\sup_{\s\in \frac{1}{\sqrt N_1}\{-1,1\}^{N_1}} \sup_{\t\in B_{N_2}^{(q)}} \frac{(\xi^{(\mu)},\s)\t_{\mu}}{\sqrt{N_1}}\geq t\right)\leq 2e^{-ct^2N_1}\,
\ee
where $c>0$ depends only on the distribution of $\xi_1^{(1)}$. 
\end{lemma}

\begin{proof}
We introduce the transpose patterns $\tilde\xi^{(i)}$ by $\tilde\xi^{(i)}_\mu:=\xi_i^{(\mu)}$, $i\in[N_1]$, $\mu\in[N_2]$.
First of all we note that
\be\label{eq:notethat}
\frac{1}{N_1^{\max(0,\frac{q-2}{2q})}}\sup_{\s\in \frac{1}{\sqrt N}\{-1,1\}^{N_1}} \sum_{i\in[N_1]}\sum_{\mu\in[N_2]}\frac{\xi_i^{(\mu)}\s_i\t_\mu}{\sqrt{N_1}}=\frac1{N_1}\sum_{i\in[N_1]}\frac{\left|(\tilde\xi^{(i)},\t)\right|}{N_1^{\max(0,\frac{q-2}{2q})}}\,.
\ee
Moreover, since for all $\t\in B_{N_2}^{(q)}$
\be\label{eq:normatau}
\|\t\|_2\leq N_2^{\max(0,\frac{q-2}{2q})}\,,
\ee
we have
\be
P\left(\frac{1}{N_1^{\max(0,\frac{q-2}{2q})}}\left|(\tilde\xi^{(i)},\t)\right|\geq t\right)\leq 2e^{-\frac{t^2N^{\max(0,\frac{q-2}{q})}_1}{2N_2^{\max(0,\frac{q-2}{q})}\|\xi^{(1)}_1\|^2_{\psi_2}}}\,.
\ee
The r.h.s. of (\ref{eq:notethat}) is the sum of independent sub-Gaussian random variables with 
$$
E[|(\tilde\xi^{(1)},\t)|]\leq \sqrt{E[|(\tilde\xi^{(1)},\t)|^2]}\lesssim N_2^{\max(0,\frac{q-2}{2q})}\,,
$$ 
thus for all $\t\in B_{N_2}^{(q)}$ for $t\gtrsim1$ we have
\be
P\left(\frac1{N_1}\sum_{i\in[N_1]}\frac{\left|(\tilde\xi^{(i)},\t)\right|}{N_1^{\max(0,\frac{q-2}{2q})}}\geq t\right)\leq 2e^{-\frac{t^2N^{1+\max(0,\frac{q-2}{q})}_1}{2N_2^{\max(0,\frac{q-2}{q})}\|\xi^{(1)}_1\|^2_{\psi_2}}}\,
\ee
for some $c>0$ depending only on the distribution of $\xi_1^{(1)}$. 

Next we cover $B^{(q)}_{N_2}$ with a number of balls in $\R^{N_2}$ with some small radius $\e>0$. For $p\geq2$ we can use Euclidean balls and the Sudakov inequality gives a bound on the minimal number $N(B^{N_2}_q,\e B^{N_2}_2)$ of such balls
$$
N(B^{N_2}_q,\e B^{N_2}_2)\leq e^{cN_2^{\frac2p}}\,
$$
(here we used that for a Gaussian vector $g$ $E[\max_{\t\in B^{N_2}_q}(\t,g)]=E[\|g\|_p]\simeq N_2^{\frac1p}$). 

For $p\in(1,2)$ we use $\ell_q$-balls and we have
$$
N(B^{N_2}_q,\e B^{N_2}_q)\leq e^{cN_2}
$$
(in the two estimates above the constants $c$ depends on $\e$ in a way we do not keep track of).

Assume now $(1+\a)^{\frac1p}\lesssim_p t$ (this is to take into account also the behaviour for small $\a$). By the union bound
for $p\geq2$ we get
\be
P\left(\frac{1}{N_1^{\max(0,\frac{q-2}{2q})}}\sup_{\s\in \frac{1}{\sqrt N}\{-1,1\}^{N_1}} \sup_{\t\in B_{N_2}^{(q)}} (\widehat\xi\s,\t)\geq t\right)\leq 2e^{cN^{\frac2p}_2-cN_1t^2}\leq 2e^{-ct^2N_1}\,.
\ee

Similarly for $p\in(1,2)$
\be
P\left(\frac{1}{N_1^{\frac{q-2}{2q}}}\sup_{\s\in \frac{1}{\sqrt N}\{-1,1\}^{N_1}} \sup_{\t\in B_{N_2}^{(q)}} (\widehat\xi\s,\t)\geq t\right)\leq 2e^{cN_2-c\left(\frac{N_1}{N_2}\right)^{\frac{q-2}{q}}N_1t^2}\leq 2e^{-ct^2N_1}\,.
\ee
\end{proof}

\begin{proof}[Proof of Proposition \ref{Prop:groundstate}]
The role of hidden variables at zero temperature is played by duality:
\be\label{eq:duality}
\sum_{\mu\in[N_2]}|(\xi^{(\mu)},\s)|^p= \left|\sup_{\t\in B_{N_2}^{(q)}} \sum_{\mu\in[N_2]}(\xi^{(\mu)},\s)\t_\mu\right|^p\,,\qquad p\geq 1\,. 
\ee
Therefore (here we shorten $\wh\xi:=\xi/\sqrt {N_1}$)
\bea
\inf_{\s\in \{-1,1\}^{N_1}} H^{(p)}(\s;\xi)&=&-\sup_{\s\in \frac{1}{\sqrt {N_1}}\{-1,1\}^{N_1}}\left|\frac{1}{N_1^{\max(0,\frac{q-2}{2q})}} \sup_{\t\in B_{N_2}^{(q)}} \sum_{\mu\in[N_2]}(\widehat\xi^{(\mu)},\s)\t_\mu\right|^p\nn\\
&=&- \left|\frac{1}{N_1^{\max(0,\frac{q-2}{2q})}}\sup_{\s\in \frac{1}{\sqrt {N_1}}\{-1,1\}^{N_1}}\sup_{\t\in B_{N_2}^{(q)}} \sum_{\mu\in[N_2]}(\widehat\xi^{(\mu)},\s)\t_\mu\right|^p\,\label{eq:bysymmetry}
\eea
by symmetry. 
It suffices to focus on the quantity inside the modulus above, which is dealt in Lemma \ref{lemma:Psup}. We have for all $t\gtrsim_p(1+\a)^{\frac1p}$
\be\label{eq:Psup-bis}
P\left(\frac{1}{N_1^{\max(0,\frac{q-2}{2q})}}\sup_{\s\in \frac{1}{\sqrt N}\{-1,1\}^{N_1}} \sup_{\t\in B_{N_2}^{(q)}} (\widehat\xi\s,\t)\geq t\right)\leq 2e^{-ct^2N_1}\,.
\ee
Combining (\ref{eq:bysymmetry}) with the bound above 
we obtain the statement. 
\end{proof}

Now we show that the patterns have energy of the same order in $\a$ of the global minimum, even though we can already observe a difference between the models with $p\geq 2$ and $p<2$. We deal with $\pm1$ binary patterns for simplicity, but a similar argument can be easily repeated for symmetric patterns with minor modifications. We have
$$
H^{(p)}(\xi^{(1)};\xi)=-\frac{\|\xi^{(1)}\|_2^p}{N_1^{1+p-p/p_+}}-\sum_{j=2}^{N_2}\frac{|(\xi^{(1)},\xi^{(j)})|^p}{N_1^{1+p-p/p_+}}\,. 
$$
This quantity concentrates around its average as $N_1$ grows:
\begin{lemma}\label{lemma:Hconcentra}
Take $t>0$ uniformly in $N_1$, small enough. It holds
\be\label{eq:Hconcentra}
P\left(\left|H^{(p)}(\xi^{(1)};\xi)-E[H^{(p)}(\xi^{(1)};\xi)]\right|\geq t\right)\leq  \begin{cases}
\exp\left(-cN_1 t^2\right)&p\in(1,2]\\
\exp\left(-cN_1 t^{\frac{2}{p}}\right)&p>2
\end{cases}
\ee
\end{lemma}
\begin{proof}
We write
$$
H^{(p)}(\xi^{(1)};\xi)=-\frac{\|\xi^{(1)}\|_2^p}{N_1^{1+p-p/p_+}}-\sum_{\mu=2}^{N_2}\frac{|(\xi^{(1)},\xi^{(\mu)})|^p}{N_1^{1+p-p/p_+}}\,.
$$
It suffices to focus on the second summand on the r.h.s. above. We have by independence
\bea
&&P\left(\left|\sum_{\mu=2}^{N_2}\frac{|(\xi^{(1)},\xi^{(\mu)})|^p}{N_1^{\frac p2}}-E\left[\frac{|(\xi^{(1)},\xi^{(\mu)})|^p}{N_1^{\frac p2}}\right]\right|\geq tN_1^{1+p/2-p/p_+}\right)\nn\\
&=&P\left(\left|\sum_{\mu=2}^{N_2}\frac{|(1,\xi^{(\mu)})|^p}{N_1^{\frac p2}}-E\left[\frac{|(1,\xi^{(\mu)})|^p}{N_1^{\frac p2}}\right]\right|\geq tN_1^{1+p/2-p/p_+}\right)\,,\label{eq:Mpestimato}
\eea
where $1$ is the constant vector with all entries equal to $1$. The random variables 
$$
T_p^{(\mu)}:=\frac{|(1,\xi^{(\mu)})|^p}{N_1^{\frac p2}}-E\left[\frac{|(1,\xi^{(\mu)})|^p}{N_1^{\frac p2}}\right]\,,\qquad \mu=2,\ldots, N_2
$$
are i.i.d. with $\|T_p^{(\mu)}\|_{\psi_{\frac 2p}}\simeq 1$. Thus by Proposition \ref{prop:Hoeffging-GEN} with $\ell=2/p$ we have
\be
\eqref{eq:Mpestimato}\leq 
\begin{cases}
\exp\left(-cN_1\min(t^2, t^{\frac{2}{p}})\right)&p\in(1,2]\\
\exp\left(-cN_1\min(t^2N_1^{p-2}, t^{\frac{2}{p}})\right)&p>2
\end{cases}
\ee
and the proof is complete. 
\end{proof}

Moreover we have
\be\label{eq:E-Hp-xi1}
|E[H^{(p)}(\xi^{(1)};\xi)]|\leq \frac{N_1^p+N_1^{\frac p2}N_2}{N_1^{1+p-p/p_+}}\simeq -\a\frac p2\Gamma\left(\frac p2\right)-\frac{1}{N_1^{1-p/p_+}} \,.
\ee
In fact
\bea
E[|(\xi^{(1)},\xi^{(\mu)})|^{p}]&=&E[\int d\l P(\xi^{(\mu)}\,:\,|(\xi^{(1)},\xi^{(\mu)})|\geq \l^{1/p})]\nn\\
&\leq& E[\|\xi^{(1)}\|_1^p]\int_0^\infty d\l e^{-\l^{\frac2p}}=\frac p2\Gamma\left(\frac p2\right) N_1^{\frac p2}\label{eq:proceeding}\,. 
\eea
Therefore
$$
H^{(p)}(\xi^{(1)};\xi)\simeq -\frac p2\Gamma\left(\frac p2\right)\a-\frac{1}{N_1^{1-p/p_+}}
$$
with very high probability. We see that if $p\geq 2$ this value is really of the same order of the ground state, while if $p\in(1,2)$ for $\a$ small and $N_1$ large the patterns have higher energy.


\section{Retrieval for $p\geq2$}\label{sect:retrieval}

In this section we prove Theorem \ref{TH1}. We look at all configurations reachable from $\xi^{(\mu)}$ by $\lfloor rN_1\rfloor$ flips and compare the energy of the pattern with the minimal energy of such configurations. By symmetry of the patterns we can reduce to look at $\mu=1$ and we may and will assume that $1$, \ie the vector with entries all equal to one, lies in $\widehat S^{N_1-1}_{1,\lfloor rN_1\rfloor}$.

Without further explanation, we introduce some more notations. 
For any point $\s\in\{-1,1\}^{N_1}$ and subset of indices $J\subseteq [N_1]$ we set
\be\label{defXY}
X_J^{(\mu)}(\s):=\frac{1}{\sqrt {N_1}}(\xi^{(\mu)}_J,\s_J)\,,\quad Y_J^{(\mu)}(\s):=\frac{1}{\sqrt {N_1}}(\xi^{(\mu)}_{J^c},\s_{J^c})\,
\ee
and $X_J^{(\mu)}(1)=:X_J^{(\mu)}$, $Y_J^{(\mu)}(1)=:Y_J^{(\mu)}$. 
We conveniently let
\be\label{eqF0}
\Phi_p(x,y):=|x+y|^p-|x-y|^p\,,\qquad \bar \Phi_p(r):=\Phi_p(r,1-r)=1-(1-2r)^p>0\,
\ee
(recall that we consider $r\in(0,1/2)$). 
We have the following useful representation (recall the definition of the {\sc flip} operator $F_J$ in Section \ref{sect:nota}). 

\begin{lemma}\label{lemma:repr-inutile}
Let $r\in(0,\frac12)$, $J\subset[N_1]$ with $|J|=\lfloor rN_1\rfloor$. It is
\be\label{eq:formulaF}
H^{(p)}(\xi^{(1)})-H^{(p)}(F_J\xi^{(1)})=-\frac{1}{N_1^{\frac {p_+-p}{2}}} \bar \Phi_p(r)-\frac{1}{N_1^{\frac {p_+}2}}\sum_{\mu=2}^{N_2} \Phi_p(X_J^{(\mu)}(\xi^{(1)}), Y_J^{(\mu)}(\xi^{(1)}))\,.
\ee
\end{lemma}
\begin{proof}
Compute
\bea
H^{(p)}(\s)-H^{(p)}(F_J\s)&=&-\frac{1}{N_1^{\kappa(p)}}\sum_{\mu\in[N_2]}|(\xi^{(\mu)}_J,\s_J)+(\xi^{(\mu)}_{J^c},\s_{J^c})|^p - |-(\xi^{(\mu)}_J,\s_J)+(\xi^{(\mu)}_{J^c},\s_{J^c})|^p\nn\\
&=&-\frac{1}{N_1^{\kappa(p)-\frac p2}} \sum_{\mu\in[N_2]}\Phi_p(X_J^{(\mu)}(\s), Y_J^{(\mu)}(\s))\,,\nn
\eea
by the definitions (\ref{defXY}), (\ref{eqF0}). We have
$
\kappa(p)-\frac p2=1+\frac p2-\frac{p}{p_+}=\frac{p_+}{2}\,.
$
Take now $v=\xi^{(1)}$. An easy computation gives
\be
X_J^{(1)}(\xi^{(1)})=\frac{\|\xi^{(1)}_J\|_2^2}{\sqrt N_1}=\frac{|J|}{\sqrt N_1}=r\sqrt{N_1}\,,\quad Y_J^{(\mu)}(\xi^{(1)})=\frac{\|\xi^{(1)}_{J^c}\|_2^2}{\sqrt {N_1}}=\frac{|J^c|}{\sqrt {N_1}}=(1-r)\sqrt{N_1}\,. 
\ee
Thus 
$$
\frac{1}{N_1^{\frac {p_+}{2}}} \Phi_p(X_J^{(1)}(\xi^{(1)}), Y_J^{(1)}(\xi^{(1)}))= \frac{1}{N_1^{\frac {p_+-p}{2}}} \Phi_p(r, 1-r)
$$
and \eqref{eq:formulaF} follows.
\end{proof}

The necessary tail estimates in order to prove Theorem \ref{TH1} are given in the next lemmas. 

\begin{lemma}\label{lemma:XYsubG}
Let $r\in(0,\frac12)$, $J\subset[N_1]$ with $|J|=\lfloor rN_1\rfloor$.
$\{X_J^{(\mu)}\}_{\mu\in[N_2]\setminus \{1\}}$ and $\{Y_J^{(\mu)}\}_{\mu\in[N_2]\setminus \{1\}}$ are independent sub-Gaussian random variables, independent one from each other, with
\be\label{eq:XYsubG}
\|X_J^{(\mu)}\|_{\psi_2}\leq \sqrt{\frac{3r}{2}}\,\qquad \|Y_J^{(\mu)}\|_{\psi_2}\leq \sqrt{\frac{3(1-r)}{2}}\,.
\ee
Moreover $\{\Phi_p(X_J^{(\mu)}, Y_J^{(\mu)})\}_{\mu\in[N_2]}$ are i.i.d. $\psi_{2/p}$ r.vs with
\be\label{eq:disF}
\|\Phi_p(X_J^{(\mu)}, Y_J^{(\mu)})\|^{\frac 2p}_{\psi_{2/p}}\leq 3\sqrt{r(1-r)}\,.
\ee
\end{lemma}
\begin{proof}
The proof of \eqref{eq:XYsubG} is standard. We proceed only for $X^{(\mu)}$, as for $Y^{(\mu)}$ is similar.  We set $\tilde\l:=\l\sqrt{N_1/2\lfloor rN_1\rfloor}$ and $\widetilde X^{(\mu)}:=(\xi^{(\mu)}_J,1)/\sqrt{\lfloor rN_1\rfloor}$. We let also $g\sim\mathcal N(0,1)$ and $\bar\xi$ be a symmetric Bernoulli $\pm1$ variable, whose expectation values are denoted by $E_g$ and $E_{\bar \xi}$. We have
\bea
&&E[e^{\frac{|X_J^{(\mu)}|^2}{\l^2}}]=E[e^{\frac{|\widetilde X^{(\mu)}|^2}{\tilde\l^2}}]=EE_g[e^{\frac{g\widetilde X^{(\mu)}}{\tilde \l}}]=E_g\left[\left(E_{\bar\xi}[e^{\frac{g\bar\xi}{\tilde\l\sqrt{\lfloor rN_1\rfloor}}}]\right)^{\lfloor rN_1\rfloor}\right]\nn\\
&=&E_g\left[e^{\lfloor rN_1\rfloor\log\cosh\left(\frac{g}{\tilde\l\sqrt{\lfloor rN_1\rfloor}}\right)}\right]\leq E_g[e^{\frac{g^2}{2\tilde\l^2}}]=(1-\tilde\l^2)^{-\frac12}
\eea
Since
$$
(1-\tilde\l^2)^{-\frac12}=\left(1-\frac12\l^2\frac{N_1}{\lfloor rN_1\rfloor}\right)^{-\frac12}<2
$$
for $\l<\sqrt\frac{3r}{2}$ we recover the first one of (\ref{eq:XYsubG}). 

To prove \eqref{eq:disF} we bound
\be
E[e^{\left(\frac{\Phi_p}{t}\right)^{\frac2p}}]\leq E[e^{\frac{(|X_J^{(\mu)}|+|Y_J^{(\mu)}|)^{2}}{t^{\frac2p}}}]\leq \frac12 E[e^{\frac{2|X_J^{(\mu)}|^{2}}{t^{\frac2p}}}]+\frac12 E[e^{\frac{2|Y_J^{(\mu)}|^{2}}{t^{\frac2p}}}]\,,
\ee
whence
$$
\|\Phi_p(X_J^{(\mu)}, Y_J^{(\mu)})\|^{\frac 2p}_{\psi_{2/p}}\leq 2\|X_J^{(\mu)}\|_{\psi_2}\|Y_J^{(\mu)}\|_{\psi_2}\,. 
$$ 
\end{proof}

We shorten in the next statement $P_{\xi}(\cdot)=P(\cdot\,|\,\xi)$.

\begin{lemma}\label{lemma:crux1}
Let $r\in(0,\frac12)$, $p\geq 2$, $t=t(r):= 1-(1-2r)^p-r^\frac p2$. 
Take any $\s\in \widehat S^{N_1-1}_{1,\lfloor rN_1\rfloor}$. 
If 
\be\label{eq:conditalpha}
\a\geq 3^{1-p}N_1^{\frac{p-2}{2}}\left(\frac{r}{1-r}\right)^{\frac{p-1}{2}}\,,
\ee
then
\be\label{eq:stimap<2-2}
P_{\xi^{(1)}}\left(H^{(p)}(\xi^{(1)})-H^{(p)}(\s)\geq -t\right)\leq \exp\left(-\frac{N_1^{\frac p2}}{24\a }\left(\frac{r}{1-r}\right)^\frac p2\right)\,
\ee
and otherwise
\be\label{eq:stimap<2-1}
P_{\xi^{(1)}}\left(H^{(p)}(\xi^{(1)})-H^{(p)}(\s)\geq -t\right)\leq \exp\left(-\frac1{24} N_1 \sqrt{\frac{r}{1-r}}\right)\,.
\ee
\end{lemma}
\begin{remark}\label{RMK}
Thinking $N_1$ very large, with an abuse of notation we will say in the sequel that a property occurs for all $\a>0$ in case it does for all $\a\gtrsim N_1^{-x}$ for some $x>0$. 
Therefore if $r>0$ uniformly in $N_1$, \ie we flip a number of bits proportional to $N_1$, we have for $p>2$ the tail (\ref{eq:stimap<2-1}) for all $\a>0$ and for $\a\lesssim \sqrt r$ for $p=2$. A sub-linear number of flips corresponds to take $r\simeq N_1^{-x}$ for some $x\in[0,1]$ (modulo log-corrections, see below). In this case we see that if $x<\frac{p-2}{p-1}$ the estimate \eqref{eq:stimap<2-1} still holds for any $\a>0$, while otherwise we have \eqref{eq:stimap<2-2}. 
\end{remark}

\begin{proof}
In Lemma \ref{lemma:tecnico} it is proven $t(r)>0$ for any $r\in(0,1/2)$.
It is clear that any $\s\in \widehat S^{N_1-1}_{1,\lfloor rN_1\rfloor}$ can be written as $F_J\xi^{(1)}$ for some index set $J$ of $\lfloor rN_1\rfloor$ elements (indeed $
J=\{i\in[N_1]\,:\,\s_i\neq \xi_i^{(1)}\}$).
Then by \eqref{eq:formulaF} we have 
\bea
&&P_{\xi^{(1)}}\left(H^{(p)}(\xi^{(1)})-H^{(p)}(\s)\geq -t\right)\nn\\&=&P_{\xi^{(1)}}\left(-\sum_{\mu=2}^{N_2} \Phi_p(X_J^{(\mu)}(\xi^{(1)}), Y_J^{(\mu)}(\xi^{(1)}))\geq N_1^{\frac {p}{2}} (\bar \Phi_p(r)-t)\right)\nn\\
&=&P_{\xi^{(1)}}\left(-\sum_{\mu=2}^{N_2} \Phi_p(X_J^{(\mu)}, Y_J^{(\mu)})\geq (N_1r)^{\frac {p}{2}}\right)\nn\\
&=&P\left(-\sum_{\mu=2}^{N_2} \Phi_p(X_J^{(\mu)}, Y_J^{(\mu)})\geq (N_1r)^{\frac {p}{2}}\right)\,,\label{eq;cont}
\eea
because of independence of the patterns and $\Phi_p(r,1-r)\geq0$. 

Note that $y>0$ is equivalent to $0\leq t<\bar \Phi_p(r)$. 
By Lemma \ref{lemma:XYsubG} $\{\Phi_p(X_J^{(\mu)}, Y_J^{(\mu)})\}_{\mu\in[N_2]}$
are centred i.i.d. r.vs which fit the assumptions of Proposition \ref{prop:Hoeffging-GEN} below (with $\ell=2/p\in(0,1]$).
Therefore
\be\label{eq:p>2}
\eqref{eq;cont}\leq \exp\left(-\frac1{24}\min\left(\frac{N^{p}_1r^\frac p2}{3^{p-1}N_2 (1-r)^\frac p2},\frac{N_1\sqrt r }{\sqrt{1-r}}\right)\right)\,.
\ee
The value of this minimum depends on $\a$. We take the first term if (\ref{eq:conditalpha}) is fulfilled,
otherwise we take the second one. 
\end{proof}

Now we are ready for the main proof. 

\begin{proof}[Proof of Theorem \ref{TH1}]
We shorten
\be\label{eq:defD}
\mathtt D_{\mu,N_1}(r_0):=\{\s\in \widehat B^{N_1}_{\mu,\lfloor r_0N_1 \rfloor}\,:\, H^{(p)}(\xi^{(\mu)};\xi)\geq H^{(p)}(\s;\xi)\}
\ee
and note that since for any $\mu\in[N_2]$
$$
\mathtt{dLM}_{N_1}^{(\mu)}\cap \widehat B^{N_1}_{\mu,\lfloor r_0N_1 \rfloor}\subseteq \mathtt D_{\mu,N_1}(r_0)\,,
$$
it is
\be\label{eq:1disth1}
P\left(\forall \mu\in[N_2]\,\,\,\,\mathtt{dLM}_{N_1}^{(\mu)}\cap \widehat B^{N_1}_{\mu,\lfloor r_0N_1 \rfloor}\subseteq \widehat B^{N_1}_{\mu,R}\right)\geq 
P\left(\forall \mu\in[N_2]\,\,\,\,\mathtt D_{\mu,N_1}(r_0)\subseteq \widehat B^{N_1}_{\mu,R}\right)\,.
\ee
Let us introduce the sets
\be
\Barr_{N_1,N_2,p}(n):=\left\{\min_{\mu\in[N_2]}\min_{\s\in\widehat S^{N_1-1}_{\mu,n}}H^{(p)}(\s)-H^{(p)}(\xi^{(\mu)})\geq t(n) \right\}\,
\ee
on which the minimal energy gap of $n$ flips from the patterns is a given $t(n)$. Write now for $r\in(0,\frac12)$ $n=\lfloor rN_1\rfloor$ and $t(n)=t(r)=t$. 
We take some $r_0\in(0,1/2)$ to be specified later.
Then, bearing in mind (\ref{eq:1disth1}), the crux is
\bea
&&P\left(\forall \mu\in[N_2]\,\,\,\,\mathtt D_{\mu,N_1}(r_0)\subseteq \widehat B^{N_1}_{\mu,R}\right)
\geq P\left(\bigcap_{n=\lfloor R\rfloor}^{\lfloor r_0N_1 \rfloor}\Barr_{N_1,N_2,p}(n)\right)\nn\\
&\geq&1-\sum_{n= \lfloor R\rfloor }^{\lfloor r_0N_1 \rfloor} P(\Barr^{c}_{N_1,N_2,p}(n))
\geq 1-N_1\min_{\lfloor R\rfloor\leq n\leq \lfloor r_0N_1 \rfloor} P(\Barr^{c}_{N_1,N_2,p}(n))\,. \label{eq:very-impo}
\eea
By the standard estimate (\ref{eq:standard-bound}) and the union bound we have
\bea
&&P\left(\Barr^c_{N_1,N_2,p}(n)\right)=P\left(\min_{\mu\in[N_2]}\min_{\s\in \widehat S^{N_1-1}_{\mu,n}}H^{(p)}(\s)-H^{(p)}(\xi_\mu)\leq t \right)\nn\\
&\leq& N_2\exp(N_1 
S(r))E\left[\sup_{\s\in \widehat S^{N_1-1}_{1,\lfloor rN_1\rfloor}}P_{\xi^{(1)}}\left(H^{(p)}(\xi^{(1)})-H^{(p)}(\s)\geq-t \right)\right]\,.\nn
\eea

The probabilities appearing in the last line are evaluated using Lemma \ref{lemma:crux1} with the same choice $t=1-(1-2r)^p-r^{\frac p2}$.

Let us first deal with $p>2$. We take $r_0\in(0,\frac12]$ such that for all $r\in[0,r_0]$ it is $25S(r)\leq \sqrt{r/(1-r)}$ and $t(r)=1-(1-2r)^p-r^{\frac p2}$ increases. Bearing in mind Remark \ref{RMK}, we let $x_p:=\frac{p-2}{p-1}$ and consider different regimes. If $n>\lfloor N_1^{1-x_p}\rfloor$ then 
Lemma \ref{lemma:crux1} yields for all $\a>0$
\be\label{eq:combinalo-o-o}
P\left(\Barr^c_{N_1,N_2,p}(n)\right)\leq N_2\exp\left(N_1 \left(
S(r)-\frac{\sqrt r}{24\sqrt{1-r}}\right)\right)\leq N_2e^{-c\sqrt{r}N_1}\simeq N_2 e^{-c\sqrt{nN_1}}\,.
\ee
Thus
\be\label{eq:pallaextp>2}
\min_{\lfloor N_1^{1-x_p}\rfloor < n\leq \lfloor r_0N_1 \rfloor} P(\Barr^{c}_{N_1,N_2,p}(n))\leq N_2 e^{-cN^{1-\frac{x_p}{2}}_1}\,.
\ee
For $n<\lfloor N_1^{1-x_p}\rfloor$ Lemma \ref{lemma:crux1} gives for all $\a>0$ 
\be\label{eq:u-u-u-usala}
P\left(\Barr^c_{N_1,N_2,p}(n)\right)\leq N_2e^{N_1
S(r)-\frac{N_1^{\frac p2}r^{\frac p2}}{24\a}}\leq N_2e^{n|\log N_1|-\frac{n^{\frac p2}}{24\a}}\,.
\ee
Thus for all $\e>0$ sufficiently small
\be\label{eq:pallaextp>2-2}
\min_{\lfloor N_1^{2\e}\rfloor<n<\lfloor N_1^{1-x_p}\rfloor} P(\Barr^{c}_{N_1,N_2,p}(n))\leq N_2e^{-cN_1^{\e p}}\,.
\ee
Moreover by \eqref{eq:u-u-u-usala} we see that also a poly-log number of flips is allowed:
\be\label{eq:pallaextp>2-3}
\min_{\lfloor (\log N_1)^{\frac2{p-2}}\rfloor<n\leq\lfloor N_1^{2\e}\rfloor} P(\Barr^{c}_{N_1,N_2,p}(n))\leq N_2e^{-c(\log N_1)^{1+\frac{2}{2-p}}}\,.
\ee
Finally we look at $n\simeq N_1^{1-x_p}$. In this case we have to fix some $\a_0>0$ and we use for $\a\leq \a_0$ the bound (\ref{eq:combinalo-o-o}) and for $\a>\a_0$ the bound (\ref{eq:u-u-u-usala}). We have
\be\label{eq:pallaextp>2-fin}
\min_{n\simeq\lfloor N_1^{1-x_p}\rfloor} P(\Barr^{c}_{N_1,N_2,p}(n))\leq N_2 e^{-cN^{1-\frac{x_p}{2}}_1}\,.
\ee

Combining (\ref{eq:very-impo}) with $R=(\log N_1)^{\frac2{p-2}}$ and (\ref{eq:pallaextp>2}), \eqref{eq:pallaextp>2-2}, \eqref{eq:pallaextp>2-3}, \eqref{eq:pallaextp>2-fin} we get
\be\label{eq:final-retr-p>2}
P\left(\forall \mu\in[N_2]\,\,\,\,\mathtt{dLM}_{N_1}^{(\mu)}\cap \widehat B^{N_1}_{\mu,\lfloor r_0N_1 \rfloor}\subseteq \widehat B^{N_1}_{\mu,\lfloor (\log N_1)^{\frac2{p-2}}\rfloor}\right)\geq 1-N_1N_2e^{-c(\log N_1)^{1+\frac{2}{2-p}}}
\ee
whence (\ref{eq:TH1-p>2}) follows. 

Now we look at $p=2$ and take $r_0=\frac 38$, so that $t(r)$ is increasing for $r\in(0,\frac38)$ (the number $3/8$ carries no special meaning). If $\a\leq\min\left(\frac13\sqrt \frac{r}{1-r},\frac{\sqrt r}{25S(r)}\right)$ then Lemma \ref{lemma:crux1} gives
\be\label{eq:bound0-retr-p=2}
P\left(B^c_{N_1,N_2,p}(n)\right)\leq N_2\exp\left(N_1 \left(
S(r)-\frac{\sqrt r}{24\a\sqrt{1-r}}\right)\right)\leq N_2 e^{-c\sqrt{nN_1}}\,.
\ee
Thus by (\ref{eq:very-impo}) with $R=\lfloor rN_1 \rfloor$ and $r\in(0,\frac38)$
\be\label{eq:caso.p=2}
P\left(\forall \mu\in[N_2]\,\,\,\,\mathtt{dLM}_{N_1}^{(\mu)}\cap \widehat B^{N_1}_{\mu,\lfloor r_0N_1 \rfloor}\subseteq \widehat B^{N_1}_{\mu,\lfloor rN_1 \rfloor}\right)\geq 1-N_1N_2e^{-\frac14\sqrt{r}N_1}\,,
\ee
whence the $p=2$ part of Theorem \ref{TH1} follows.
\end{proof}


\section{Absence of retrieval for $p\in(1,2]$}\label{sect:absence}
 
In this section we present the proof of Theorem \ref{TH2}.

We set for brevity for $\mu\in[N_2]\setminus \{1\}$, $p\in(1,2]$, $k\in[N_1]$, $J\subseteq[N_1]$, $\s\in\{-1,1\}^{N_1}$
\be\label{eq:W}
W_{p,k,J}^{(\mu)}(\s):=\frac{2p}{\sqrt{N_1}}\xi_k^{(\mu)}v_k\sign(Z_{J}^{(\mu)}(\s))|Z_{J}^{(\mu)}(\s)|^{p-1}
\ee
where (recall the definition of the {\sc flip} operator $F_J$ in section \ref{sect:nota})
\be\label{eq:defZ}
Z_J^{(\mu)}(\s):=\frac{1}{\sqrt{N_1}}(\xi^{(\mu)},F_J\s)\,.
\ee

Next we give the central technical lemma employed in the proof of Theorem \ref{TH2}. In the sequel we shorten $J+k:=J\cup\{k\}$ if $k\notin J$ and $J-k:=J\setminus\{k\}$ if $k\in J$.
\begin{lemma}\label{lemma:reprH-absence}
Let $r\in(0,\frac12]$, $J\subseteq[N_1]$ with $|J|=\lfloor rN_1\rfloor$. 
For any $p\in(1,2]$ we have
\bea
N_1(H^{(p)}(F_{J_{\pm k}}\xi^{(1)})-H^{(p)}(F_J\xi^{(1)}))&=&\mp \sum_{\mu=2}^{N_2} W_{p,k,J}^{(\mu)}(\xi^{(1)})-\a\varsigma N_1^{1-\frac p2} \nn\\
&\pm&\frac{2p(1-2r)^{p-1}}{N_1^{1-\frac p2}}+\OO{N_1^{2-\frac p2}}\label{eq:reprH-absence}
\eea
where $\varsigma$ is a strictly positive and uniformly bounded random variable depending on $\{\xi^{(1)}_k\xi^{(\mu)}_k, Z_{J}^{(\mu)}(\xi^{(1)})\}_{\mu=2...N_2}$. Setting 
\be\label{eq:defdp}
d(p):= 2^p\left(2p-1-2^{p-1}\left(\frac{p-1}{p}\right)^{p-1}\frac{3p-2}{p}\right)\,,
\ee
we have $C^p>\varsigma\geq d(p)$ for any realisation of $\varsigma$ and $C>0$ an absolute constant. 

In particular
\be\label{eq:reprH-absence-p=2}
\frac{N_1}{4}(H^{(2)}(F_{J_{\pm k}}\xi^{(1)})-H^{(2)}(F_J\xi^{(1)}))=\mp\sum_{\mu\geq2}\frac{\xi_k^{(1)}\xi_k^{(\mu)}}{\sqrt{N_1}}Z_J^{(\mu)}(\xi^{(1)})-\a\pm(1-2r)\mp\frac{1}{N_1}\,.
\ee
\end{lemma}
\begin{proof}
By Lemma \ref{lemma:repr-inutile} we have for any $k\notin J$
\bea
N_1(H^{(p)}(F_{J_{+k}}\xi^{(1)})-H^{(p)}(F_J\xi^{(1)}))&=&N_1(H(F_{J_{+ k}}\xi^{(1)})-H(\xi^{(1)})-(H(F_J\xi^{(1)})-H(\xi^{(1)})))\nn\\
&=&N_1^{\frac p2}\left(\bar \Phi_p\left(r+\frac1{N_1}\right)-\bar \Phi_p(r)\right)\nn\\
&+&\sum_{\mu=2}^{N_2} \Phi_p(X_{J_{+k}}^{(\mu)}(\xi^{(1)}), Y_{J_{+k}}^{(\mu)}(\xi^{(1)}))\nn\\
&-&\sum_{\mu=2}^{N_2} \Phi_p(X_{J}^{(\mu)}(\xi^{(1)}), Y_{J}^{(\mu)}(\xi^{(1)}))\nn\\
&=&N_1^{\frac p2}\left(|1-2r|^p-\left|1-2r-\frac2{N_1}\right|^p\right)\nn\\
&+&\sum_{\mu=2}^{N_2}\left( |Z_J^{(\mu)}(\xi^{(1)})|^p-\left|Z_{J}^{(\mu)}(\xi^{(1)})-2\frac{\xi_k^{(\mu)}\xi_k^{(1)}}{\sqrt{N_1}}\right|^p \right)\,.\label{eq:repr-H+}
\eea
Similarly for all $k\in J$
\bea
N_1(H^{(p)}(F_{J_{-k}}\xi^{(1)})-H^{(p)}(F_J\xi^{(1)}))&=&N_1^{\frac p2}\left(|1-2r|^p-\left|1-2r+\frac2{N_1}\right|^p\right)\nn\\
\!\!\!\!\!&+&\!\!\!\!\!\sum_{\mu=2}^{N_2}\left( |Z_J^{(\mu)}(\xi^{(1)})|^p-\left|Z_{J}^{(\mu)}(\xi^{(1)})+2\frac{\xi_k^{(\mu)}\xi_k^{(1)}}{\sqrt{N_1}}\right|^p\right)\label{eq:repr-H-}
\eea
For $p=2$ a straightforward computation gives (\ref{eq:reprH-absence-p=2}) from (\ref{eq:repr-H+}) and (\ref{eq:repr-H-}).

In general for $p\in(1,2)$ we have to use Taylor expansion. Let $(p)_0:=1$ and $(p)_k:=\prod_{j=0}^{k-1}(p-j)$ for $k\geq1$. Assuming $r\in(0,\frac12)$, $N_1$ large enough (\ie $N_1(1-2r)>2$) we have
\be
\left|1-2r\pm\frac2{N_1}\right|^p=|1-2r|^p\pm\frac{2p}{N_1}|1-2r|^{p-1}+\frac{4}{N_1^2}\sum_{k\geq2}\frac{(p)_k}{k!}\frac{2^{k-2}(1-2r)^{p-k}}{N_1^{k-2}}
\ee
and using $|(p)_k|\leq k!$ we get
\be
\left|\frac{4}{N_1^2}\sum_{k\geq2}\frac{(p)_k}{k!}\frac{2^{k-2}(1-2r)^{p-k}}{N_1^{k-2}}\right| \leq \frac{4}{N_1^2(1-2r)^{2-p}}\sum_{k\geq0}\frac{2^k}{((1-2r)N_1)^{k}}\lesssim\frac{1}{N_1^2(1-2r)^{2-p}}\,.
\ee
Therefore
\be\label{eq:bound-pezzo1}
N_1^{\frac p2}\left(|1-2r|^p-\left|1-2r\pm\frac2{N_1}\right|^p\right)=\mp\frac{2p}{N_1^{1-\frac p2}}|1-2r|^{p-1}+ \OO{N_1^{2-\frac p2}}\,. 
\ee
On the other hand, for $r=\frac12$ this correction term is trivially of order $N_1^{-\frac p2}$. 

Recall now \eqref{eq:W} and compute
\bea
&&\left|Z_{J}^{(\mu)}(\xi^{(1)})\pm2\frac{\xi_k^{(\mu)}\xi_k^{(1)}}{\sqrt{N_1}}\right|^p\nn\\
&=&1_{\{|Z_{J}^{(\mu)}(\xi^{(1)})|\geq 2N_1^{-\frac12}\}}\Big( |Z_{J}^{(\mu)}(\xi^{(1)})|^p\pm W_{p,k,J}^{(\mu)}(\xi^{(1)}) \nn\\
&+& \sum_{\ell\geq 2}\frac{(p)_\ell}{\ell!}\frac{|Z_{J}^{(\mu)}(\xi^{(1)})|^p}{(Z_{J}^{(\mu)}(\xi^{(1)}))^\ell} \frac{(\pm 2\xi^{(\mu)}_k\xi^{(1)}_k)^\ell}{N_1^{\frac \ell2}} \Big)\nn\\
&+&1_{\{|Z_{J}^{(\mu)}(\xi^{(1)})|< 2N_1^{-\frac12}\}}\frac{2^p}{N_1^{\frac p2}}\Big(1\pm \frac p2\xi_k^{(\mu)}\xi_k^{(1)}Z_{J}^{(\mu)}(\xi^{(1)})\sqrt{N_1}\nn\\
&+&\sum_{\ell\geq 2}\frac{(p)_\ell}{\ell!} \left(\pm\frac{\xi_k^{(\mu)}\xi_k^{(1)}}{2}Z_{J}^{(\mu)}(\xi^{(1)})\sqrt{N_1}\right)^\ell\Big)\nn\,. 
\eea
Thus
\bea
&&\left( |Z_J^{(\mu)}(\xi^{(1)})|^p-\left|Z_{J}^{(\mu)}(\xi^{(1)})\pm2\frac{\xi_k^{(\mu)}\xi_k^{(1)}}{\sqrt{N_1}}\right|^p\right)=\mp W_{p,k,J}^{(\mu)}(\xi^{(1)})\nn\\
&-&1_{\{|Z_{J}^{(\mu)}(\xi^{(1)})|\geq 2N_1^{-\frac12}\}}\sum_{\ell\geq 2}\frac{(p)_\ell}{\ell!}\frac{|Z_{J}^{(\mu)}(\xi^{(1)})|^p}{(Z_{J}^{(\mu)}(\xi^{(1)}))^\ell} \frac{(\pm2\xi^{(\mu)}_k\xi^{(1)}_k)^\ell}{N_1^{\frac \ell2}}\label{eq:riga2}\\
&-&1_{\{|Z_{J}^{(\mu)}(\xi^{(1)})|< 2N_1^{-\frac12}\}}\left(\frac{2^p}{N_1^\frac p2}\pm\frac{2^p}{N_1^\frac p2}\frac{\xi^{(\mu)}_k\xi^{(1)}_k}{2}Z_{J}^{(\mu)}(\xi^{(1)})\sqrt{N_1}- |Z_{J}^{(\mu)}(\xi^{(1)})|^p\mp W_{p,k,J}^{(\mu)}(\xi^{(1)})\right)\nn\\\label{eq:riga1}\\
&-&1_{\{|Z_{J}^{(\mu)}(\xi^{(1)})|< 2N_1^{-\frac12}\}}\frac{2^p}{N_1^{\frac p2}}\sum_{\ell\geq 2}\frac{(p)_\ell}{\ell!} 2^{-\ell}(\pm Z_{J}^{(\mu)}(\xi^{(1)})\xi_k^{(\mu)}\xi_k^{(1)}\sqrt{N_1})^\ell\,.\label{eq:riga3}
\eea

The contributions (\ref{eq:riga2}) and (\ref{eq:riga3}) are very similar and will be dealt together. 
Using that for $\ell\geq2$ it is $\ell-p>0$ and $\ell^2(p)_\ell\leq 2p^2\ell!$ we have
\be\label{eq:bound-riga2}
|\eqref{eq:riga2}|\leq \frac{2p^2}{N_1^{\frac p2}}\sum_{\ell\geq2} \ell^{-2}\,.
\ee
Similarly
\be\label{eq:bound-riga3}
|\eqref{eq:riga3}|\leq \frac{2^{p+1}p^2}{N_1^{\frac p2}}\sum_{\ell\geq2} \ell^{-2}\,.
\ee

Furthermore, depending on the value of $\sign(Z_{J}^{(\mu)}(\xi^{(1)}))\xi^{(1)}_k\xi^{(\mu)}_k$ we can have either
\bea
&&\eqref{eq:riga2}=1_{\{|Z_{J}^{(\mu)}(\xi^{(1)})|\geq 2N_1^{-\frac12}\}}\frac{1}{N_1^{\frac p2}}\sum_{\ell\geq 2}\frac{(p)_\ell}{\ell!}\frac{|Z_{J}^{(\mu)}(\xi^{(1)})|^p}{|Z_{J}^{(\mu)}(\xi^{(1)})|^\ell} \frac{2^\ell}{N_1^{\frac \ell2}}\nn\\
&=&1_{\{|Z_{J}^{(\mu)}(\xi^{(1)})|\geq 2N_1^{-\frac12}\}}\frac{|Z_{J}^{(\mu)}(\xi^{(1)})|^p}{N_1^{\frac p2}}\left[(1+2|Z_{J}^{(\mu)}(\xi^{(1)})\sqrt N_1|^{-1})^p-(1+2p|Z_{J}^{(\mu)}(\xi^{(1)})\sqrt N_1|^{-1})\right]\geq0\,,\nn\\ 
\eea
where equality is achieved only if $p=1$, or
\bea
&&\eqref{eq:riga2}=1_{\{|Z_{J}^{(\mu)}(\xi^{(1)})|\geq 2N_1^{-\frac12}\}}\frac{1}{N_1^{\frac p2}}\sum_{\ell\geq 2}(-1)^\ell\frac{(p)_\ell}{\ell!}\frac{|Z_{J}^{(\mu)}(\xi^{(1)})|^p}{|Z_{J}^{(\mu)}(\xi^{(1)})|^\ell} \frac{2^\ell}{N_1^{\frac \ell2}}\nn\\
&=&1_{\{|Z_{J}^{(\mu)}(\xi^{(1)})|\geq 2N_1^{-\frac12}\}}\frac{|Z_{J}^{(\mu)}(\xi^{(1)})|^p}{N_1^{\frac p2}}\sum_{\ell\geq 1}\left(\frac{\sqrt{N_1}Z_{J}^{(\mu)}(\xi^{(1)})}{2}\right)^{-2\ell}\left(\frac{(p)_{2\ell}}{{2\ell}!}-\frac{(p)_{2\ell+1}}{{2\ell+1}!}\left|\frac{\sqrt{N_1}Z_{J}^{(\mu)}(\xi^{(1)})}{2}\right|^{-1}\right)\nn\\\label{eq:quant1}
\eea
where we split the sum over even and odd $\ell\geq2$ and rename the indices to get the second identity. The quantity in \eqref{eq:quant1} above is non-negative, since for $\ell\geq1$ $(p)_{2\ell}\geq0$ and $(p)_{2\ell+1}\leq0$, which can be shown by observing that $(p)_{\ell\geq3}=(-1)^{\ell}p(p-1)\prod_{j=2}^{\ell-1}(j-p)$ (again the equality is achieved only for $p=1$). 
Similarly we have
\bea
&&\eqref{eq:riga3}=1_{\{|Z_{J}^{(\mu)}(\xi^{(1)})|< 2N_1^{-\frac12}\}}\frac{2^p}{N_1^{\frac p2}}\sum_{\ell\geq 2}\frac{(p)_\ell}{\ell!} 2^{-\ell}| Z_{J}^{(\mu)}(\xi^{(1)})\sqrt{N_1}|^\ell\nn\\
&=&1_{\{|Z_{J}^{(\mu)}(\xi^{(1)})|< 2N_1^{-\frac12}\}}\frac{2^p}{N_1^{\frac p2}}\left[(1+p|2^{-1}Z_{J}^{(\mu)}(\xi^{(1)})\sqrt N_1|)^p-(1+|2^{-1}Z_{J}^{(\mu)}(\xi^{(1)})\sqrt N_1|)\right]\geq0\nn\\\label{eq:quant3/2}\, 
\eea
(equality is achieved only if $p=1$) or
\bea
&&\eqref{eq:riga3}=1_{\{|Z_{J}^{(\mu)}(\xi^{(1)})|< 2N_1^{-\frac12}\}}\frac{2^p}{N_1^{\frac p2}}\sum_{\ell\geq 2}(-1)^\ell\frac{(p)_\ell}{\ell!}2^{-\ell}| Z_{J}^{(\mu)}(\xi^{(1)})\sqrt{N_1}|^\ell\nn\\
&=&1_{\{|Z_{J}^{(\mu)}(\xi^{(1)})|< 2N_1^{-\frac12}\}}\frac{2^p}{N_1^{\frac p2}}\sum_{\ell\geq 1}\left(\frac{\sqrt{N_1}Z_{J}^{(\mu)}(\xi^{(1)})}{2}\right)^{2\ell}\left(\frac{(p)_{2\ell}}{{2\ell}!}-\frac{(p)_{2\ell+1}}{{2\ell+1}!}\frac{|\sqrt{N_1}Z_{J}^{(\mu)}(\xi^{(1)})|}{2}\right)\geq0\,\nn\\\label{eq:quant2}
\eea
by the same argument used for the \eqref{eq:quant1}.

We conclude 
$$
\eqref{eq:riga2}\leq0,\quad \eqref{eq:riga3}\leq0\quad\mbox{for $p\in(1,2)$}\,. 
$$ 

Moreover for any $a\in(0,1)$
\bea
\eqref{eq:quant3/2}1_{\{|Z_{J}^{(\mu)}(\xi^{(1)})|\geq 2aN_1^{-\frac12}\}}&\geq& \frac{(2a)^p}{aN_1^{\frac p2}}\label{eq:>a+}\\
\eqref{eq:quant2}1_{\{|Z_{J}^{(\mu)}(\xi^{(1)})|\geq2aN_1^{-\frac12}\}}&\geq& \frac{2^p}{N_1^{\frac p2}}a^2\left(\frac{(p)_2}{2}+a\frac{|(p)_3|}{6}\right)\label{eq:>a-}\,.
\eea

With a bit of algebra we rewrite the term in (\ref{eq:riga1}) as
\bea
&&1_{\{|Z_{J}^{(\mu)}(\xi^{(1)})|< 2N_1^{-\frac12}\}}\frac{2^p}{N_1^\frac{p}{2}}\left(\left(1-\frac{|\sqrt{N_1}Z_{J}^{(\mu)}(\xi^{(1)})|^p}{2^{p}}\right)\right.\nn\\
&\mp&\left. p\xi^{(\mu)}_k\xi^{(1)}_k\sign(Z_{J}^{(\mu)}(\xi^{(1)}))\left(\frac{|\sqrt{N_1}Z_{J}^{(\mu)}(\xi^{(1)})|^{p-1}}{2^{p-1}}-\frac{|\sqrt{N_1}Z_{J}^{(\mu)}(\xi^{(1)})|}{2}\right)\right)\,.\label{eq:bitofalgebra}
\eea
According to the value of $\xi^{(\mu)}_k\xi^{(1)}_k\sign(Z_{J}^{(\mu)}(\xi^{(1)}))$ the term inside the parenthesis can be either
$$
1-|x|^p+p(|x|^{p-1}-|x|)\quad\mbox{ or }\quad 1-|x|^p-p(|x|^{p-1}-|x|)\,,
$$
where we shortened $|x|:=\frac{|2Z_{J}^{(\mu)}(\xi^{(1)})|}{N_1^{\frac 12}}<1$. The first expression above is clearly positive, while the second one is positive thanks to Lemma \ref{lemma:tecnico}. More precisely for any $a\in(0,1)$
\be\label{eq:a3}
\eqref{eq:bitofalgebra}1_{\{|Z_{J}^{(\mu)}(\xi^{(1)})|< 2aN_1^{-\frac12}\}}\geq 1-a^p-p(a^{p-1}-a)>0\,. 
\ee
From the representation (\ref{eq:bitofalgebra}) we also get the bound
\be\label{eq:bound-algebra}
\eqref{eq:riga1}\leq \frac{C}{N_1^\frac p2}\,. 
\ee

Now we pick $a=2(p-1)/p$ into (\ref{eq:>a+}), \eqref{eq:>a-}, \eqref{eq:a3}. Combining with \eqref{eq:bound-riga2}, \eqref{eq:bound-riga3}, \eqref{eq:bound-algebra} we conclude that the lines (\ref{eq:riga2}), (\ref{eq:riga1}) and (\ref{eq:riga3}) define a random variable $\varsigma:=\varsigma(\{\xi^{(1)}_k\xi^{(\mu)}_k, Z_{J}^{(\mu)}(\xi^{(1)})\}_{\mu=2...N_2})$ lying in a uniformly bounded interval away from the origin
such that
\be\label{eq:bound-pezzo2}
\sum_{\mu=2}^{N_2}\left( |Z_J^{(\mu)}(\xi^{(1)})|^p-\left|Z_{J}^{(\mu)}(\xi^{(1)})\pm2\frac{\xi_k^{(\mu)}\xi_k^{(1)}}{\sqrt{N_1}}\right|^p\right)=\mp \sum_{\mu=2}^{N_2} W_{p,k,J}^{(\mu)}(\xi^{(1)})+ \a\varsigma N_1^{1-\frac p2}\,. 
\ee
Precisely, we have
$$
\varsigma\geq 2^p\left(2p-1-2^{p-1}\left(\frac{p-1}{p}\right)^{p-1}\frac{3p-2}{p}\right)\,.
$$ 
This and (\ref{eq:bound-pezzo1}) give (\ref{eq:reprH-absence}).
\end{proof}

Now we turn to the proof of Theorem \ref{TH2}, which we conveniently split in severals steps.

{\em Step 1: reduction.}
Due to the exchangeability of the patterns and their entries we have
\bea
&&P\left(\mathtt{LM}^{(\mu)}_{N_1}\cap \widehat B^{N_1}_{\mu, \lfloor r N_1\rfloor}\neq \emptyset\right)\nn\\
\!\!\!\!&\leq& \!\!\!\!N_2 \sum_{\ell=1}^{\lfloor rN_1\rfloor} \sum_{\substack{J\subset [N_1]\\|J|=\ell}} P\left(\bigcap_{k\notin J}\{H(F_{J_{+k}}\xi^{(1)})-H(F_{J}\xi^{(1)})>0\}\,,\,\,\bigcap_{k\in J}\{H(F_{J-k}\xi^{(1)})-H(F_{J}\xi^{(1)})>0\}\right)\,\nn\\
\!\!\!\!&=&\!\!\!\!\!\!N_2\sum_{\ell=1}^{\lfloor rN_1\rfloor}\binom{N_1}{\ell}P\left(\bigcap_{k>\ell}\{H(F_{[\ell]_{+k}}\xi^{(1)})-H(F_{[\ell]}\xi^{(1)})>0\}\,,\,\,\bigcap_{k\leq\ell}\{H(F_{[\ell]-k}\xi^{(1)})-H(F_{[\ell]}\xi^{(1)})>0\}\right)\nn\,.\\\label{eq:proof1}
\eea
Recalling (\ref{eq:W}) and (\ref{eq:defZ}) we set for brevity
\bea
W_{p,k,\ell}^{(\mu)}(\xi^{(1)})&:=&W_{p,k,[\ell]}^{(\mu)}(\xi^{(1)})\,,
\quad W_{p,k,\ell}^{(\mu)}:=W_{p,k,\ell}^{(\mu)}(1)\,,\\
M^{(\mu)}&:=&\frac1{N_1}\sum_{i\in[N_1]}\xi_i^{(\mu)}\,,\quad Q^{(\mu)}:=\sign(M^{(\mu)})|M^{(\mu)}|^{p-1}\,.
\eea
By Lemma \ref{lemma:reprH-absence} we have for $N_1$ large enough ($r':=\ell/N_1$)
\bea
\eqref{eq:proof1}&\leq&N_2\sum_{\ell=1}^{\lfloor rN_1\rfloor}\binom{N_1}{\ell}P\left(\forall k>\ell\,\,\,-\sum_{\mu=2}^{N_2} W_{p,k,\ell}^{(\mu)}(\xi^{(1)})\geq d(p)\a N_1^{1-\frac p2} -\frac{2p(1-2r')^{p-1}}{N_1^{1-\frac p2}}\,,\,\right.\nn\\
&&\left.\forall k\in[\ell] \,\,\,\sum_{\mu=2}^{N_2} W_{p,k,\ell}^{(\mu)}(\xi^{(1)})\geq d(p)\a N_1^{1-\frac p2}+\frac{2p(1-2r')^{p-1}}{N_1^{1-\frac p2}} \right)\nn\\
&=&N_2\sum_{\ell=1}^{\lfloor rN_1\rfloor}\binom{N_1}{\ell}P\left(\forall k>\ell\,\,\,-\sum_{\mu=2}^{N_2} W_{p,k,\ell}^{(\mu)}\geq d(p)\a N_1^{1-\frac p2} -\frac{2p(1-2r')^{p-1}}{N_1^{1-\frac p2}}\,,\,\right.\nn\\
&&\left.\forall k\in[\ell] \,\,\,\sum_{\mu=2}^{N_2} W_{p,k,\ell}^{(\mu)}\geq d(p)\a N_1^{1-\frac p2}+\frac{2p(1-2r')^{p-1}}{N_1^{1-\frac p2}} \right)\nn\\
&=&N_2\sum_{\ell=1}^{\lfloor rN_1\rfloor}\binom{N_1}{\ell}P\left(\forall k>\ell\,\,\,(Q,\tilde\xi^{(k)})\geq d(p)\a N_1^{1-\frac p2} -\frac{2p(1-2r')^{p-1}}{N_1^{1-\frac p2}}\,,\,\right.\nn\\
&&\left.\forall k\in[\ell] \,\,\,(Q,\tilde\xi^{(k)})\geq d(p)\a N_1^{1-\frac p2}+\frac{2p(1-2r')^{p-1}}{N_1^{1-\frac p2}} \right)
\eea
(recall that $\tilde\xi^{(k)}$ denotes the $k$-th transposed pattern). In the second identity above we have exploited independence of the pattern $\xi^{(1)}$ to replace $W_{p,k,\ell}^{(\mu)}(\xi^{(1)})$ by $W_{p,k,\ell}^{(\mu)}$ and in the third one the independence of the first $\ell$ entries from all the others and the flip-symmetry to replace $Z_{\ell}^{(\mu)}$ by $M^{(\mu)}$. We also used that the independent random variables $Q^{(2)}\,\ldots Q^{(N_2)}$ are symmetric. We denote by $Q$ (respectively $M$) the vector whose $\mu$-th component is $Q^{(\mu)}$ (respectively $M^{(\mu)}$). 

{\em Step 2: disentangling by the FKG inequality.}
Let $\mathcal M$ be the $\s$-field generated by $M^{(2)},\ldots, M^{(N_2)}$ and notice that $Q^{(2)}\,\ldots Q^{(N_2)}$ are $\mathcal M$-measurable. We shorten $P_{\mathcal M}(\cdot):=P(\cdot\,|\,\mathcal M)$.  
The crucial observation here (first remarked in \cite{lou} for the Hopfield model) is that for each $\mu\in[N_2]$ the law of $\xi^{(\mu)}$ conditionally on  $\mathcal M$ is a permutation distribution (as the increments of a simple random walk given the position). 
Therefore the FKG inequality applies (see for instance \cite{brics}) and we have
\bea
&&\!\!\!\!\!\!\!\!\!\!\!\!\!\!\!\!\!\!P_{\mathcal M}\left(\forall k>\ell\,\,\,(Q,\tilde\xi^{(k)})\geq d(p)\a N_1^{1-\frac p2} -\frac{2p(1-2r')^{p-1}}{N_1^{1-\frac p2}}\right.\,,\nn\\
&&\left.\,\forall k\in[\ell] \,\,\,(Q,\tilde\xi^{(k)})\geq d(p)\a N_1^{1-\frac p2}+\frac{2p(1-2r')^{p-1}}{N_1^{1-\frac p2}} \right)\nn\\
&\leq&\prod_{k>\ell}P_{\mathcal M}\left((Q,\tilde\xi^{(k)})\geq d(p)\a N_1^{1-\frac p2} -\frac{2p(1-2r')^{p-1}}{N_1^{1-\frac p2}}\right)\nn\\
&&\prod_{k\in[\ell]}P_{\mathcal M}\left((Q,\tilde\xi^{(k)})\geq d(p)\a N_1^{1-\frac p2} +\frac{2p(1-2r')^{p-1}}{N_1^{1-\frac p2}}\right)\,.\label{eq:continua}
\eea
Note that $E[\xi_k^{(\mu)}\,|\,\mathcal M]=M^{(\mu)}$ and $\Var[\xi_k^{(\mu)}\,|\,\mathcal M]=1-(M^{(\mu)})^2$.
Therefore introducing $\overline Q\in\R^{N_2}$ with components $\overline Q^{(\mu)}:=Q^{(\mu)}\sqrt{1-(M^{(\mu)})^2}$, we get
\be\label{eq:Hoeff}
P_{\mathcal M}\left((Q,\tilde\xi^{(k)})\geq t\right)\leq \exp\left( -\frac{(t-\|M\|_p^p)^2}{2\|\overline Q\|_2^2} \right)\,,
\ee
by the Hoeffding inequality. Note that this quantity is independent on $k$ and also we have the simple bound $\|\overline Q\|_2^2\leq \| Q\|_2^2=\|M\|^{2p-2}_{2p-2}$. Then (note $d(2)=1$)
\bea
&&\eqref{eq:continua}_{p\in(1,2)}\leq \exp\left( -N_1\frac{(\a d(p)N_1^{1-\frac p2}-\|M\|_p^p)^2}{2\|M\|^{2p-2}_{2p-2}} \right)\,,\label{eq:rhs1}\\
&&\eqref{eq:continua}_{p=2}\leq \exp\left( -\frac{N_1}{2\|M\|^{2}_{2}}\left(r'(\a-(1-2r')-\|M\|_2^2)^2+(1-r')(\a+(1-2r')-\|M\|_2^2)^2\right) \right)\,.\nn\\\label{eq:rhs2} 
\eea
{\em Step 3: concentration.}
Now we have to take the global expectations of the r.h.s. above. We notice $E[\|M\|_p^p]\simeq\a N_1^{1-\frac p2}$ (this is an identity for $p=2$) and  $\|(M^{(\mu)})^p-E[(M^{(\mu)})^p]\|_{\psi_{\frac2p}}\simeq N_1^{-\frac p2}$. 
We can give precise upper bounds for these quantities. It holds for any $p>0$
\be\label{eq:bounde(p)}
E[\|M\|_p^p]\leq \frac{\a N_1}{N_1^{\frac p2}}\int e^{-\frac{x^{\frac 2p}}{2}}dx=:\a N_1^{1-\frac p2} e(p)\,
\ee
(to prove it, proceed as in the computation giving \eqref{eq:proceeding}) and
\be
\|(M^{(\mu)})^p-E[(M^{(\mu)})^p]\|_{\psi_{\frac2p}}\leq \left(\frac{3}{2N_1}\right)^\frac p2
\ee
(to prove it, proceed as in the proof of Lemma \ref{lemma:XYsubG}). 

Hence by Proposition \ref{prop:Hoeffging-GEN} in Appendix \ref{Sect:Tail} (with $\ell=2/p$)
\be\label{eq:conc-Mp}
P\left(\left|\sum_{\mu\geq2} |M^{(\mu)}|^p-E[|M^{(\mu)}|^p]\right|\geq t \right)\leq 2\exp\left(-\frac18\min\left(\frac{2^pt^2N_1^{p-1}}{3^p\a},\frac{2t^{\frac2p}N_1^{2-\frac2p}}{3\a^{\frac 2p-1}}\right)\right)\,. 
\ee
We will also use the following sub-Gaussian estimate, which follows from \cite[Corollary 2.8]{talasup} (there the constant was not specified, but our choice is however not the optimal one). For any $p\in(1,2)$ there is a number $h>0$ such that for any $t< 2\a h N_1^{2-p}$

\be\label{eq:conc-M(p-1)}
P\left(\left|\sum_{\mu\geq2} |M^{(\mu)}|^{2p-2}-E[|M^{(\mu)}|^{2p-2}]\right|\geq t \right)\leq 2\exp\left( -\frac{t^2 }{4\a h N_1^{3-2p}}\right)\,. 
\ee

A sketch of the proof of \eqref{eq:conc-M(p-1)} is given at the end of this section. Note that for $p=2$ (\ref{eq:conc-M(p-1)}) reduces to the standard Gaussian estimate in the Bernstein inequality \eqref{eq:conc-Mp}$|_{p=2}$ (however numerical constants may change a bit). 

{\em Step 4: finalising the argument for $p\in(1,2)$.}
Using (\ref{eq:conc-Mp}) with $t= \frac12\a N_1^{1-\frac p2}|d(p)-e(p)|=:\t_p$ we obtain
\bea
E[\mbox{r.h.s. of \eqref{eq:rhs1}}]&\leq &E\left[\mbox{r.h.s. of \eqref{eq:rhs1}}1_{\{|\|M\|_p^p-E[\|M\|_p^p]|\leq \t_p\}}\right]+P\left(|\|M\|_p^p-E[\|M\|_p^p]|\geq \t_p\right)\nn\\
&\leq &E\left[\exp\left( -\frac{\a^2N_1^{3-p}(d(p)-e(p))^2}{\|M\|^{2p-2}_{2p-2}} \right)\right]\nn\\
&+&2\exp\left(-\frac18\min\left(\frac{2^p\t_p^2N_1^{p-1}}{3^p\a},\frac{2\t_p^{\frac2p}N_1^{2-\frac2p}}{3\a^{\frac 2p-1}}\right)\right)\nn\\
&=&E\left[\exp\left( -\frac{\a^2N_1^{3-p}(d(p)-e(p))^2}{\|M\|^{2p-2}_{2p-2}} \right)\right]+2e^{-\a N_1 \kappa_1(p)}
\label{eq:final-bound1}\,,
\eea
where
\be\label{eq:kappa1}
\kappa_1(p):=\frac18\min\left(\frac14 \left(\frac{2}{3}\right)^p (d(p)-e(p))^2,\frac{2}{3}2^{-\frac2p}(d(p)-e(p))^{\frac 2p} \right)\,. 
\ee

Using \eqref{eq:conc-M(p-1)} with $t=\frac12E[\|M\|_{2p-2}^{2p-2}]$ and \eqref{eq:bounde(p)} (with $p\to 2p-2$) we obtain
\bea
&&E\left[\exp\left( -\frac{\a^2N_1^{3-p}(d(p)-e(p))^2}{\|M\|^{2p-2}_{2p-2}} \right)\right]\nn\\
&\leq& E\left[\exp\left( -\frac{2\a^2N_1^{3-p}(d(p)-e(p))^2}{3\a N^{2-p}e(2p-2)} \right)\right]+2\exp\left( -\frac{\a N_1}{16 h}\right)\nn\\
&=&2\exp\left( -\frac{2\a N_1(d(p)-e(p))^2}{2e(2p-2)}\right)+2\exp\left( -\frac{\a N_1}{16 h}\right)\,.
\eea

Combining the display above with (\ref{eq:proof1}), (\ref{eq:continua}), (\ref{eq:rhs1}), (\ref{eq:final-bound1}) and using (\ref{eq:standard-bound}) we have 
\bea
P\left(\mathtt{LM}^{(\mu)}_{N_1}\cap \widehat B^{N_1}_{\mu, \lfloor r N_1\rfloor}\neq\emptyset\right)&\leq& 2N_2\sum_{\ell=1}^{\lfloor rN_1\rfloor}\binom{N_1}{\ell}\left(e^{-\a N_1 \kappa_1(p)}+2e^{ -\a N_1\frac{2(d(p)-e(p))^2}{2e(2p-2)}}+e^{ -\frac{\a N_1}{16 h}}\right)\nn\\
&\leq& 6r N_2N_1 e^{N_1\left(S(r)-\a\min\left( \kappa_1(p), \frac{2(d(p)-e(p))^2}{2e(2p-2)},\frac{1}{16 h}\right)\right)}\,, 
\eea
which yields the assertion for $p\in(1,2)$.

{\em Step 5: finalising the argument for $p=2$.} Using (\ref{eq:conc-Mp})$_{p=2}$ with $t=(1+\a)(1-2r')=:\t$ in the first line of the display below and (\ref{eq:conc-M(p-1)})$_{p=2}$ with $t=\frac12E[\|M\|_{2}^{2}]$ in the third line we have
\bea
E[\mbox{r.h.s. of \eqref{eq:rhs2}}]&\leq &E\left[\exp\left( -\frac {N_1}2\frac{((1-2r')-\t)^2}{\|M\|^{2}_{2}} \right)\right]+P\left(|\|M\|_2^2-E[\|M\|_2^2]|\geq \t\right)\nn\\
&\leq&E\left[\exp\left( -\frac {N_1}2\frac{\a^2(1-2r')^2}{\|M\|^{2}_{2}} \right)\right]+
2e^{-\frac23N_1(1+\a)(1-2r')\min\left(1,\frac{2(1+\a)(1-2r')}{3\a}\right)}\nn\\
&\leq &e^{-\frac13\a N_1(1-2r')^2}+2e^{-\frac13\a N_1}+2e^{-\frac23 N_1(1+\a)(1-2r')\min\left(1,\frac{2(1+\a)(1-2r')}{3\a}\right)}\nn\\
&\leq &e^{-\frac13\a N_1(1-2r)^2}+2e^{-\frac13\a N_1}+2e^{-\frac23 N_1(1+\a)(1-2r)\min\left(1,\frac{2(1+\a)(1-2r)}{3\a}\right)}\label{eq:final-bound2}\,,
\eea
as $r'$ ranges from $1/N_1$ to $r$. 
So combining (\ref{eq:proof1}), (\ref{eq:continua}), (\ref{eq:rhs1}), (\ref{eq:final-bound2}) and using again (\ref{eq:standard-bound}) we obtain
\bea
P\left(\mathtt{LM}^{(\mu)}_{N_1}\cap \widehat B^{N_1}_{\mu, \lfloor r N_1\rfloor}\neq\emptyset\right)&\leq& N_2\sum_{\ell=1}^{\lfloor rN_1\rfloor}\binom{N_1}{\ell}\left(\mbox{terms in }\eqref{eq:final-bound2}\right)\nn\\
&\leq& 2r N_2N_1 \left(e^{N_1(S(r)-\frac13\a(1-2r)^2)}+e^{N_1(S(r)-\frac13\a)}\right.\nn\\
&+&\left.e^{N_1\left(S(r)-\frac23(1+\a)(1-2r)\min\left(1,\frac{2(1+\a)(1-2r)}{3\a}\right)\right)}\right)\,.\nn
\eea
The first two summands are negative if $\a>3S(r)/(1-2r)^2$. The fact that also the third one is so is verified in Lemma \ref{lemma:verify}, Appendix \ref{Sect:AppTH2}. The proof is complete. 

\begin{proof}[Proof of \eqref{eq:conc-M(p-1)} (sketch)]
If $p<2$ it is $2p-2<2$. We bound
\bea
E[e^{\left|\frac{1}{\sqrt N_1}\sum_{i\in[N_1]}\xi_i\right|^{2p-2}}]&\leq&2\sum_{n\geq 1}E[1_{\{n-1\leq \frac{1}{\sqrt N_1}\sum_{i\in[N_1]}\xi_i\leq n\} }e^{\left|\frac{1}{\sqrt N_1}\sum_{i\in[N_1]}\xi_i\right|^{2p-2}}]\nn\\
&\leq &2\sum_{n\geq 1} e^{n^{2p-2}-\frac{(n-1)^2}{4}}=:h<\infty\,.
\eea

We can write
$$
|M^{(\mu)}|^{2p-2}=\frac{1}{N_1^{p-1}}\left(\frac{1}{\sqrt N_1}\sum_{i\in[N_1]}\xi^{(\mu)}_i\right)^{2p-2}\,.
$$
If $s\in(0,1)$ we have (see for instance \cite[Lemma 2.6]{talasup})
\be
E[e^{s |M^{(\mu)}|^{2p-2}}]\leq e^{s^2 h}\,. 
\ee
Thus by the Markov inequality and optimisation over $s$
\be
P\left(N_1^{p-1}\left|\sum_{\mu\geq2} |M^{(\mu)}|^{2p-2}-E[|M^{(\mu)}|^{2p-2}]\right|\geq t \right)\leq e^{s^2\a N_1 h-s t}\leq e^{-\frac{t^2}{4\a h N_1}}\,. 
\ee
provided $t\leq 2h\a N_1$. Changing variables implies the assertion. 
\end{proof}


\appendix

\section{Tail estimates}\label{Sect:Tail}
Here we present tail estimates for sums of i.i.d. r.vs used in the main text. The following statement is not new and we give the proof here mainly for the reader's convenience. In fact the proof of the subsequent formula (\ref{eq:Lpgauss}) for $\ell\in[1,2]$ is classical and can be found for instance in \cite[Corollaries 2.9, 2.10]{talasup} (though the formulation is slightly different there). So we focus on the case $\ell\in(0,1)$. For similar statements, see \cite[Proposition 3.2]{newman} and \cite[Theorem 6.21]{L-T}. 

\begin{proposition}\label{prop:Hoeffging-GEN}
Let $\ell\in(0,2]$, $X_1,\dots,X_N$ i.i.d. r.vs. 
Then for $N$ large enough
\be\label{eq:Lpgauss}
P\left(\left|\sum_{i\in[N]}X_i\right|\geq t\right)\leq 2
\exp\left(-\frac18\min\left(\frac{t^2}{\|X_1\|^2_{\psi_\ell}N},\frac{t^{\ell}}{\|X_1\|^\ell_{\psi_\ell}N^{\max(\ell-1,0)}}\right)\right)\,.
\ee
\end{proposition}

\begin{proof}[Proof (only for $\ell\in(0,1)$)]

We have by assumption
\be\label{eq:stimaX1}
P\left(|X_1|\geq t\right)
\leq e^{-\frac{t^\ell}{2\|X_1\|^\ell_{\psi_\ell}}}\,. 
\ee
Let now $s:=\|X_1\|_{\psi_\ell}N^{\frac{1}{2-\ell}}$ and set
\be
X^s_i:= X_i1_{\{|X_i|<s\}}\,. 
\ee
Then we have
\bea
P\left(\left|\sum_{i\in[N]}X_i\right|\geq t\right)&\leq&P\left(\sum_{i\in[N]}X_i\geq t\,,\quad\, \sup_{i\in[N]}|X_i|<s\right)+P\left(\sup_{i\in[N]}|X_i|\geq s\right)\nn\\
&\leq&P\left(\left|\sum_{i\in[N]}X^s_i\right|\geq t\right)+e^{-\frac{s^{\ell}}{4\|X_1\|^\ell_{\psi_\ell}}}\label{eq:dec-P}\,.
\eea

Set now 
$$
\bar\mu:=\frac{1}{4 s^{1-\ell}\|X_1\|^\ell_{\psi_\ell}}\,.
$$
We note that for any $0\leq\mu\leq\bar\mu$ (and $N$ large enough) it is
$$
\mu X^s_i\leq \frac{|X^s_i|^\ell}{A\|X_i\|^\ell_{\psi_\ell}}\,.
$$
Using the bound $x^2\leq e^{\frac{|x|^\ell}{A}}$
we compute
\bea
E[e^{\mu X^s_i}]&=&1+\mu^2\|X_i^s\|^2_{\psi_\ell}\sum_{n\geq0}\mu^n\frac{E[(X^s_i)^{n+2}]}{\|X_i\|^2_{\psi_\ell}(n+2)!}\nn\\
&\leq&1+\mu^2\|X_i\|^2_{\psi_\ell} E\left[e^{\frac{|X^s_i|^\ell}{10\|X_i\|^\ell_{\psi_\ell}}}\sum_{n\geq0}\frac{1}{n!} \left(\frac{|X^s_i|^\ell}{10\|X_i\|^\ell_{\psi_\ell}}\right)^n\right]\nn\\
&\leq&1+\mu^2\|X_i\|^2_{\psi_\ell}E\left[e^{\frac{|X_i|^\ell}{5\|X_i\|^\ell_{\psi_\ell}}}\right]\nn\\
&\leq&\exp \left(\mu^2\|X_i\|^2_{\psi_\ell}E\left[e^{\frac{|X_i|^\ell}{5\|X_i\|^\ell_{\psi_\ell}}}\right]\right)\leq \exp (2\mu^2\|X_i\|^2_{\psi_\ell})\,. \label{eq:expbound/new}
\eea
It follows that
\be\label{eq:cnor2}
P\left(\left|\sum_{i\in[N]}X^s_i\right|\geq t\right)\leq 2e^{-\mu t+2N\mu^2\|X_1\|^2_{\psi_\ell}}\leq
\begin{cases}
2\exp\left(-\frac{t^2}{8\|X_1\|^2_{\psi_\ell}N}\right)&0<t< 4N\bar\mu\|X_1\|^2_{\psi_\ell}\,;\\
2\exp\left(-\bar\mu t+2N\bar\mu^2\|X_1\|^2_{\psi_\ell}\right)&t\geq 4N\bar\mu\|X_1\|^2_{\psi_\ell}\,. 
\end{cases}
\ee
With our choice of parameters the above formula rewrites as
\be\label{eq:cnor2-II}
P\left(\left|\sum_{i\in[N]}X^s_i\right|\geq t\right)\leq
\begin{cases}
2\exp\left(-\frac{t^2}{8\|X_1\|^2_{\psi_\ell}N}\right)&0<t< s\,;\\
2\exp\left(-\frac{t^\ell}{8\|X_1\|^\ell_{\psi_\ell}}\right)&t\geq s\,. 
\end{cases}
\ee
Combining (\ref{eq:dec-P}) and (\ref{eq:cnor2-II}) gives the assertion. 
\end{proof}


\section{Two technical lemmas}\label{Sect:AppTH2}

The following two results are basically calculus. 

\begin{lemma}\label{lemma:tecnico}
Let $g(x,p):=1-(1-2x)^p-x^{\frac p2}$ and $f(x,p):=1-x^p-px^{p-1}+px$. It is $g>0$ for all $p\geq2$ and $x\in[0,1/2]$. Moreover for any $a\in(0,1)$ it is $f(p,x)\geq f(p,a)>0$ for all $p\in(1,2]$ and $x\in[0,a]$. 
\end{lemma}
\begin{proof}
For $x\in[0,1/2]$ the function $(1-2x)^p+x^{\frac p2}$ is decreasing in $p$, so it suffices to study $g(x,2)$ for which one verifies explicitly $g(x,2)> 0$ for all $x\in[0,1/2]$.

Now we pass to $f$. First we note that
\be\label{eq:firstT}
1+(x+p)\log x\leq x\,,\qquad\forall\, x\in[0,1]\,.
\ee
The proof is simple: we compare the function $\log x$ with $\frac{x-1}{x+p}$ for $x\in[0,1]$ and since 
$$
\frac 1x=\frac {d}{dx}\log x\geq \frac {d}{dx} \frac{x-1}{x+p}=\frac{p+1}{(x+p)^2}\qquad\forall\, x\in[0,1]
$$
and in $x=1$ the two functions intersect, (\ref{eq:firstT}) follows. 

Next we note that $f(x,1)=0$ and $f(x,2)\geq 0$ for all $x\in[0,1]$. Then we show that $f$ is non-decreasing in $p$ uniformly in $x\in[0,1]$. We compute
$$
\frac {\partial}{\partial p} f(x,p)=x\left(1-x^{p-2}\left(1+(x+p)\log x\right)\right)\geq x\left(1-x^{p-1}\right)\geq0\,
$$
thanks to (\ref{eq:firstT}). This tells us $f\geq0$. Moreover we compute
$$
\frac {\partial}{\partial x} f(x,p)= p\left(1-x^{p-1}\left(\frac{p-1}{x}+1\right)\right)\,.
$$
We have for all $x\in[0,1]$
$$
\frac{p-1}{x}+1\geq\frac{1}{x^{p-1}}\,.
$$
The above inequality is clearly true if $x$ is near the origin and at $x=1$. Indeed it must hold in the whole interval $[0,1]$, since the functions on both sides are decreasing.

It follows that $f$ is decreasing in $[0,1]$ uniformly in $p\in(1,2]$, whence the assertion.
\end{proof}

\begin{lemma}\label{lemma:verify}
Let $r\in[0,\frac12]$, $c_1>0$. Let also $\bar r=\bar r(c_1)\in[0,1/2]$ defined implicitly by
$$
\frac{S(\bar r)}{1-2\bar r}=c_1
$$
and set $c_2:=\max(\frac{1}{c_1}, \frac{(1-2\bar r)^2}{2c_1\bar r})$. For all $\a\geq c_2S(r)/(1-2r)^2$ it holds
\be\label{eq:verify}
S(r)\leq c_1(1+\a)(1-2r)\min\left(1,\frac{(1+\a)(1-2r)}{\a}\right)\,. 
\ee
\end{lemma}
\begin{proof}
\eqref{eq:verify} selects two conditions, namely either
\be\label{eq:or}
\a\leq \frac{1-2r}{2r}\,,\quad c_1\a\geq\frac{S(r)}{1-2r}-c_1\,\qquad\mbox{ or }\qquad
\a>\frac{1-2r}{2r}\,,\quad c_1\frac{(1+\a)^2}{\a}\geq \frac{S(r)}{(1-2r)^2}\,.
\ee
For $r\in[0,1/2]$ the function $S(r)$ increases and $c_1(1-2r)$ decreases. Let us denote $\bar r$ their unique intersection point in $[0,1/2]$. Clearly $\bar r$ depends on $c_1$ and $\bar r\to0$ as $c_1\to0$. If $r\in[0,\bar r]$ then for every $\a<(1-2r)/2r$ it holds
$$
c_1\a\geq\frac{S(r)}{1-2r}-c_1\,.
$$
Moreover there is $C>0$ such that
$$
\frac{1-2r}{2r}\geq C \frac{S(r)}{(1-2r)^2}\qquad \forall r\in[0,\bar r]\,. 
$$
Indeed by definition of $\bar r$ the above condition is implied by
$$
\frac{1-2r}{2r}\geq \frac{c_1C}{1-2r}\qquad \forall r\in[0,\bar r]\,,
$$
therefore it suffices to take $c_2:=\frac{(1-2\bar r)^2}{2c_1\bar r}$ and we have the statement for $r\in[0,\bar r]$.

For $r\in[\bar r, 1/2]$ we use the second condition in \eqref{eq:or}. First we observe that, since $(1+\a)^2/\a>\a$, the condition $c_1\frac{(1+\a)^2}{\a}\geq \frac{S(r)}{(1-2r)^2}$ is implied by $\a\geq CS(r)/(1-2r)^2$ for all $C>c^{-1}_1$. It remains to show that there is $C>c^{-1}_1$ such that
\be\label{eq:above}
\frac{1-2r}{2r}\leq C\frac{S(r)}{(1-2r)^2}\quad\mbox{or equivalently}\quad \frac{(1-2r)^2}{2r}\leq C\frac{S(r)}{(1-2r)}\,.
\ee
By definition of $\bar r$ 
$$
\frac{S(r)}{(1-2r)}\geq c_1\qquad \forall r\in[\bar r, 1/2]\,.
$$
The l.h.s. of the second inequality in \eqref{eq:above} is decreasing and its r.h.s. is increasing, whence it suffices to require
$$
\frac{(1-2\bar r)^2}{2\bar r}\leq C c_1\,.
$$
Thus taking $c_2:=\max(\frac{1}{c_1}, \frac{(1-2\bar r)^2}{2c_1\bar r})$ we have proved the statement also for $r\in[\bar r, 1/2]$. 
\end{proof}


\bibliographystyle{amsplain}
\bibliography{RBMT0}

\end{document}